\documentclass[11pt]{amsart}
\usepackage{amssymb,amsmath,mathtools}
\usepackage{enumerate}
\usepackage{xcolor}
\usepackage{caption}
\usepackage[centering]{geometry}
\usepackage{graphicx,tikz,wrapfig}
\usepackage{color}
\usepackage{amsmath}
\usepackage{amssymb}
\usepackage{amscd}
\usepackage{framed}
\usepackage{bbm}
\usepackage{tcolorbox}

\usepackage{amsthm}

\newcommand{\R}{\mathbb{R}}
\newcommand{\inr}[1]{\left\langle #1 \right\rangle}

\newcommand{\N}{\mathbb{N}}

\newtheorem{theorem}{Theorem}[section]
\newtheorem{proposition}[theorem]{Proposition}
\newtheorem{lemma}[theorem]{Lemma}

\theoremstyle{definition}
\newtheorem{definition}[theorem]{Definition}

\newtheorem{example}[theorem]{Example}
\newtheorem{remark}[theorem]{Remark}

\newtheorem{question}[theorem]{Question}

\newcommand{\E}{\mathbb{E}}
\renewcommand{\P}{\mathbb{P}}
\newcommand{\V}{\mathrm{\mathbb{V}ar}}

\numberwithin{equation}{section}

\begin{document}

\title{Empirical approximation of the gaussian distribution in $\mathbb{R}^d$}
\author{Daniel Bartl}
\address{University of Vienna, Faculty of Mathematics}
\email{daniel.bartl@univie.ac.at}
\author{Shahar Mendelson}
\address{ETH Z\"urich, Department of Mathematics}
\email{shahar.mendelson@gmail.com}
\date{\today}

\begin{abstract}
Let $G_1,\dots,G_m$ be independent copies of the standard gaussian random vector in $\mathbb{R}^d$. We show that there is an absolute constant $c$ such that for any $A \subset S^{d-1}$, with probability at least $1-2\exp(-c\Delta m)$, for every $t\in\mathbb{R}$,
\[ 
\sup_{x \in A} \left| \frac{1}{m}\sum_{i=1}^m 1_{ \{\langle G_i,x\rangle  \leq t \}} - \mathbb{P}(\langle G,x\rangle \leq t) \right| 
\leq \Delta +  \sigma(t) \sqrt\Delta.
\]
Here $\sigma(t) $ is the variance of $1_{\{\langle G,x\rangle\leq t\}}$   and $\Delta\geq \Delta_0$, where $\Delta_0$ is determined by an unexpected complexity parameter of $A$ that captures the set's geometry (Talagrand's $\gamma_1$ functional). 
The bound, the probability estimate, and the value of $\Delta_0$ are all (almost) optimal.

We use this fact to show that if $\Gamma=\sum_{i=1}^m \langle G_i,x\rangle e_i$ is the  random matrix that has $G_1,\dots,G_m$ as its rows, then the structure of $\Gamma(A)=\{\Gamma x: x\in A\}$ is far more rigid and well-prescribed than was previously expected.	
\end{abstract}

\maketitle
\setcounter{equation}{0}
\setcounter{tocdepth}{1}

\section{Introduction}

Let $\mu$ be a probability measure, set $X$ to be distributed according to $\mu$ and let $X_1,\dots,X_m$ be independent copies of $X$. 
One of the key features of the random sample $(X_i)_{i=1}^m$ is that it captures many of the properties of the underlying measure $\mu$---a phenomenon that is used frequently in diverse areas of mathematics, statistics and computer science.

While there is a variety of features of $\mu$ that can be recovered using the random sample, we are interested in a rather strong notion of structure preservation.
To explain the notion we focus on, let us start by considering a real-valued random variable $X$.
Clearly, $X$ is determined by its distribution function $F(t)=\P( X\leq t)$. At the same time, the empirical distribution is a random, real-valued function, defined by
$$
F_{m}(t)=\P_m(X\leq t ) = \frac{1}{m}\sum_{i=1}^m 1_{(-\infty,t]}(X_i),
$$
and the obvious question is whether the two functions are `close' in an appropriate sense for a typical sample $(X_i)_{i=1}^m$.
One natural notion of distance between the two distribution functions is the worst deviation, that is
\begin{align}
\label{eq:sup.norm}
\sup_{t\in\R} |\P_m(X\leq t)-\P(X\leq t)|=\|F_m-F\|_{L_\infty}.
\end{align}
A high-probability estimate on \eqref{eq:sup.norm} was established in the classical \emph{Dvoretzky-Kiefer-Wolfowitz} inequality \cite{dvoretzky1956asymptotic}:
there are absolute constants $c$ and $C$ such that for every $\Delta>0$,
\begin{align}
\label{eq:dkw.thm}
\P\left( \|F_m-F\|_{L_\infty} \geq \sqrt\Delta \right)\leq C\exp(-c\Delta m).
\end{align}

\begin{remark}
The optimal choice $c=C=2$ was established in \cite{massart1990tight}, though in what follows we do not pay attention to optimising the specific values of absolute constants.
\end{remark}

While it is well-known that the estimate in the Dvoretzky-Kiefer-Wolfowitz inequality cannot be improved, it is reasonable to expect that the deviation  $| \P_m(X\leq t) - \P(X\leq t)  |$ is smaller than $\sqrt{\Delta}$ when $\P(X \leq t)$ is significantly smaller or significantly larger than $1/2$. Specifically, let
\begin{align*}
\sigma^2(t)
&= \V\left(1_{(-\infty,t]}(X) \right) 
=\P(X\leq t) (1-\P(X\leq t)),
\end{align*}
and observe that
$$
\V\left( \P_m(X\leq t) \right) = \frac{ \sigma^2(t) }{m}.
$$
It stands to reason that the typical behaviour of $| \P_m(X\leq t) - \P(X\leq t) |$ should be of the order of $\sigma(t)/\sqrt{m}$ and as it happens, that is (almost) the case.

\begin{theorem} 
\label{thm:intro.ratio.single}
There are absolute constants $C$ and $c$ such that  for  
\[  \Delta \geq C \frac{ \log \log m}{ m} ,\]
with probability at least $1-2\exp(-c\Delta m)$, for every $t\in\R$,
\begin{align}
\label{eq:intro.ratio.1dim}
& \left| \P_m(X\leq t)  - \P(X\leq t)  \right|
\leq \Delta + \sigma(t)\sqrt{\Delta  }  .
\end{align}
\end{theorem}

The proof of Theorem \ref{thm:intro.ratio.single}, which  will be referred to in what follows as the \emph{scale sensitive DKW inequality}, is based on standard methods from empirical processes theory and  can be found in \cite{bartl2023variance}.
Moreover, Theorem \ref{thm:intro.ratio.single} is the best that one can hope for:
Firstly, the restriction $\Delta\gtrsim\frac{\log\log m}{m}$ is necessary;  and secondly, if $\sigma^2(t)\geq\Delta$ then with probability at least $2\exp(-c\Delta m)$, 
\[| \P_m(X\leq t)  - \P(X\leq t) |
\geq c'\sigma(t)\sqrt\Delta.\]
The proofs of both facts can be found in \cite{bartl2023variance}; see also Appendix \ref{app:ratio} for a more detailed explanation of the optimality of Theorem \ref{thm:intro.ratio.single}.

\vspace{0.5em}
Theorem \ref{thm:intro.ratio.single} implies that the empirical distribution $F_m(t)=\P_m(X \leq t)$ inherits many of the features of the distribution of a real-valued random variable $X$. 
What is less clear is whether a similar phenomenon holds simultaneously for a collection of random variables.

More accurately, let $X$ be an isotropic random vector\footnote{That is, $X$ is centred and its covariance is the identity.} in $\R^d$ and set $X_1,\dots,X_m$ to be independent copies of $X$.
Let $A\subset S^{d-1}$ and consider the collection of linear functionals $\left\{ \inr{\, \cdot ,x} : x \in A\right\}$. For $x\in A$, set $F_{x}(t)=\P(\inr{X,x}\leq t)$ to be the distribution function of the marginal $\inr{X,x}$ and put
$$
F_{m,x}(t) = \P_m(\inr{X,x}\leq t)=\frac{1}{m}\sum_{i=1}^m 1_{(-\infty,t]}( \inr{X_i,x})
$$
to be the corresponding empirical distribution function.
Let
\begin{align*}
\sigma_x^2(t)
&=\P(\inr{X,x}\leq t) ( 1-  \P(\inr{X,x}\leq t)) .
\end{align*}
\begin{tcolorbox}
\begin{question} 
\label{qu:main-2}
For what choices of $\Delta$ is there an event of  probability at least $1-2\exp(-c\Delta m)$ on which  for every $x \in A$ and every $t \in \R$,
$$
| F_{m,x}(t) - F_x(t) | 
\leq \Delta + \sigma_x(t) \sqrt{\Delta}  \,\, ?
$$
\end{question}
\end{tcolorbox}
Theorem \ref{thm:intro.ratio.single} answers Question \ref{qu:main-2} in one  extreme case---when $|A|=1$. 
For the other extreme case, when $A=S^{d-1}$, see \cite[Theorem 3.2]{mendelson2021approximating}:

\begin{theorem} \label{thm:intro.vc}
There are absolute constants $C$ and $c$ such that the following holds. Let $X$ be an isotropic random vector. If  $m\geq d$ and
\begin{equation} \label{eq:intro.rest.Delta.vc}
\Delta \geq C \frac{d}{m}\log\left(\frac{em}{d}\right),
\end{equation}
then with probability at least $1-2\exp(-c\Delta m)$, for every $x\in S^{d-1}$ and $t\in \R$,
$$
\left| \P_m(\inr{X,x} \leq t ) - \P(\inr{X,x} \leq t )\right|
	\leq  \Delta +  \sigma_x(t) \sqrt{\Delta} .
$$
\end{theorem}

As will become clear in what follows, the only potential source of looseness (when $A=S^{d-1}$) is the logarithmic factor in \eqref{eq:intro.rest.Delta.vc}.

\begin{remark}
\label{rem:m.smaller.d}
It is not surprising that the behaviour of a single random variable is completely different from a uniform estimate in $S^{d-1}$.
For example, if $G$ is the standard gaussian random vector in $\R^d$ then for every $x \in S^{d-1}$, $\P(\inr{G,x}\leq 0 )=1/2$. And, by Theorem \ref{thm:intro.ratio.single}, with probability at least $0.9$,
\begin{align*}
\left| \P_m(\inr{G,x} \leq 0 ) - \P(\inr{G,x} \leq 0 )\right|\leq  \frac{C}{\sqrt m}.
\end{align*}
However, if  $m<d$ then  for every realization of $G_1,\dots,G_m$ there is some $x\in S^{d-1}$ for which $\inr{G_i,x}=0$ for every $1\leq i \leq m$, and in particular $\P_m(\inr{G,x}=0)=1$. 
Therefore, whenever $m<d$, with probability 1,
\begin{align*}
\sup_{x\in S^{d-1}}
\left| \P_m(\inr{G,x} \leq 0 ) - \P(\inr{G,x} \leq 0 )\right|\geq  \frac{1}{2}.
\end{align*}
\end{remark}

The answer to Question \ref{qu:main-2} is satisfactory in the two extreme cases $|A|=1$ and $A=S^{d-1}$.
At the same time, what happens when $A$ is an arbitrary subset of $S^{d-1}$ is not clear at all.
Our goal is to explore that question and for  a satisfactory answer one must identity the right notion of \emph{complexity} of $A$.

Unfortunately, it is hopeless to try and adapt the proof of Theorem \ref{eq:intro.rest.Delta.vc} to sets that are significantly smaller than the entire sphere: the proof from \cite{mendelson2021approximating} is based on \emph{\rm{VC}-theory}; specifically, on the fact that the VC-dimension of the class of indicators of halfspaces in $\R^d$ is proportional to $d$. 
Clearly, the VC-dimension is the wrong notion of complexity in the context of a general set $A \subset S^{d-1}$: it is straightforward to construct `small' sets for which the VC-dimension of the class of halfspaces $\{1_{\{ \inr{\cdot,x} \leq t \}} : x \in A, \ t \in \R \}$ is proportional to $d$.

\vspace{0.5em}
In what follows we address Question \ref{qu:main-2}  when $X$ is the standard gaussian random vector in $\R^d$. 
The  rather surprising and  (almost) optimal answer  is described in the next section.

\subsection{The scale sensitive DKW inequality for the gaussian random vector}
Let $X=G$ be the standard gaussian random vector in $\R^d$.
The natural candidate for a complexity parameter of the set $A$ in the gaussian case is the expectation of the supremum of the gaussian process indexed by $A \cup (-A)$, 
$$
\E \sup_{x \in A} |\inr{G,x}|.
$$

By Talagrand's majorizing measures theorem  (see the presentation in \cite{talagrand2022upper}) the behaviour of $\E \sup_{x\in A} \inr{G,x}$ is completely characterized by a metric invariant of the set $A\subset \R^d$.

\begin{definition}[Talagrand's $\gamma_\alpha$ functional]
 \label{def:chaining.intro}
An admissible sequence of a set $A$ is a family  $(A_s)_{s \geq 0}$  of subsets of $A$ whose cardinalities satisfy $|A_0|=1$ and $|A_s| \leq 2^{2^s}$ for $s \geq 1$. 
Denote by $\pi_s \colon A \to A_s$ the nearest point map with respect to the Euclidean distance, let $\alpha>0$ and set
$$
\gamma_\alpha(A) = \inf \sup_{x\in A} \,\, \sum_{s\geq 0 } 2^{s/\alpha} \|x-\pi_s x\|_{2},
$$
where the infimum is taken with respect to all admissible sequences of $A$.
\end{definition}

A standard chaining argument shows that there is an absolute constant $C$ such that for any $A \subset \R^d$, $\E \sup_{x \in A} \inr{G,x} \leq C \gamma_2(A)$. 
Talagrand's majorizing measures theorem implies that the reverse inequality is also true\footnote{The upper and lower bounds actually  hold for arbitrary gaussian processes rather than just for the canonical gaussian process indexed by  subsets of $\R^d$.}; hence for any $A \subset \R^d$,
$$
c \gamma_2(A) 
\leq \E \sup_{x\in A} \inr{G,x} 
\leq C \gamma_2(A).
$$
The proofs of both facts and a detailed survey on the generic chaining mechanism can be found in \cite{talagrand2022upper}.

\vspace{0.5em}

Although $\gamma_2(A \cup -A) \sim \E \sup_{x \in A} |\inr{G,x}|$ is the natural complexity parameter of the set $A$, a closer inspection of the problem at hand reveals that it is unlikely to be the right one. 
Indeed, by rotation invariance $\frac{1}{m}\sum_{i=1}^m G_i$ has the same distribution as $\frac{1}{\sqrt m} G$; hence,
$$
\E \sup_{x \in A } \left|\frac{1}{m}\sum_{i=1}^m \inr{G_i,x} - \E \inr{G,x} \right| 
= \frac{1}{\sqrt{m}} \E \sup_{x \in A} |\inr{G,x} |
\sim \frac{\gamma_2(A \cup -A)}{\sqrt{m}}.
$$
Moreover, by tail integration 
\[ \E \inr{G,x}=\E \inr{G,x}^+ - \E \inr{G,x}^-
=\int_0^\infty ( 1-F_x(t) )\,dt - \int_{-\infty}^0 F_x(t) \,dt,\]
and a similar representation holds for the empirical expectation; hence 
\begin{align*}
\sup_{x \in A } & \left|\frac{1}{m}\sum_{i=1}^m \inr{G_i,x} - \E \inr{G,x} \right|
=\sup_{x \in A }  \left| \int_{-\infty}^\infty \left( F_x(t) - F_{m,x}(t) \right)  \,dt \right|.
\end{align*}
Therefore,
\begin{align}
\label{eq:gamma2.integral}
\frac{\gamma_2(A \cup -A)}{\sqrt{m}} \sim \sup_{x \in A }  \left| \int_{-\infty}^\infty \left( F_x(t) - F_{m,x}(t) \right) \,  dt \right|
\end{align}
and as a result, $\gamma_2(A\cup-A)$ captures a different notion of cancellations that occur between $F_x$ and $F_{m,x}$ than the one we are looking for:  
following \eqref{eq:gamma2.integral} it stands to reason that $\frac{1}{m} \gamma_2(A \cup -A)$ should be  significantly smaller than a typical realization of  $\sup_{x \in A} \| F_{m,x}-F_x\|_{L_\infty}$.

We show in what follows that  $\Delta \sim \frac{1}{m}\gamma_2^2(A \cup -A)$ is far too optimistic, and that the right complexity parameter is the larger
\[
\Delta \sim  \frac{\gamma_1(A \cup -A)}{m}.
\]

To formulate the wanted estimate, set 
\begin{align*}
\sigma^2(t)
&= \P(\inr{G,x} \leq t ) (1-\P(\inr{G,x} \leq t )) \\
&=F_g(t) (1-F_g(t))
\end{align*}
where $F_g$ is the standard normal distribution function.
To ease notation, from here on we assume that $A$ is a symmetric set, that is, $A = -A$; in particular, $\gamma_1(A) \geq 1$.

The following is our main result.

\begin{tcolorbox}
\begin{theorem} 
\label{thm:intro.gaussian}
There are absolute constants $C$ and $c$ such that  the following holds. 
Let $A \subset S^{d-1}$ be symmetric. 
Set $m\geq \gamma_1(A)$ and put
\begin{equation} \label{eq:intro.restriction.delta}
\Delta \geq C \frac{\gamma_1(A)}{m} \log^3\left(\frac{em}{\gamma_1(A)}\right).
\end{equation}
Then with probability at least $1-2\exp(-c\Delta m)$, for every $x\in A$ and $t\in\R$,
\begin{equation} \label{eq:intro.gaussian.error}
\left| \P_m(\inr{G,x} \leq t ) - \P(\inr{G,x} \leq t )\right| \leq   \Delta + \sigma(t) \sqrt{\Delta } .
\end{equation}
\end{theorem}
\end{tcolorbox}

Theorem \ref{thm:intro.gaussian} is clearly optimal for $\Delta$ that satisfies the restriction \eqref{eq:intro.restriction.delta}, simply because \eqref{eq:intro.gaussian.error} is the best one can expect even for a single random variable.
The only question is whether the restriction on $\Delta$ can be improved, and we will show in Section \ref{sec:Sudakov} that  (essentially) it cannot.

Before doing so, let us present how Theorem \ref{thm:intro.gaussian} can be used to describe the geometry of images of a gaussian matrix.
Specifically, set
\[ \Gamma = \sum_{i=1}^m \inr{G_i, \cdot} e_i \colon \R^d\to\R^m\]
to be the random matrix whose rows are  $G_1,\dots,G_m$; thus, $\Gamma$ has iid  standard gaussian entries.

\subsection{The coordinate structure of $\Gamma (A)$}

As it happens, Theorem \ref{thm:intro.gaussian} implies that the set
\[ \Gamma (A) = \left\{ (\inr{G_i,x})_{i=1}^m :  x\in A \right\} \] 
has a surprisingly well-prescribed structure.
To formulate the result, for $1\leq i\leq m-1$, let 
\[\lambda_i= F_g^{-1}\left(\frac{i}{m}\right)\]
and put $\lambda_m=\lambda_{m-1}$.
The behaviour of  $\lambda_i$ is well understood; for example, it follows from \cite[Lemma 5.2]{csiszar2011information} that for $u \in(0,\frac{1}{2}]$,
\[F^{-1}_g(u)= - \sqrt{2\log\left(\frac{1}{u}\right) -\log\log\left(\frac{1}{u}\right) - O(1) } \,,\]
and with the obvious modifications  a similar estimate holds for $u\in(\frac{1}{2},1)$.

For a vector $v\in\R^m$, denote by $v^\sharp$ its non-decreasing rearrangement, i.e.\ $v_1=\min_i v_i$ and $v_m=\max_i v_i$.
Thanks to Theorem \ref{thm:intro.gaussian}, with high probability, the set  $\{(\Gamma x)^\sharp : x\in A\}$  is close (in Hausdorff distance) to the set  $\{(\lambda_i)_{i=1}^m\}$:

\begin{tcolorbox}
\begin{theorem}
\label{thm:wasserstein.matrix}
	There are absolute constants $C_1,c_2,C_3$ such that the following holds.
	Let $A \subset S^{d-1}$ be symmetric. 
	Set $m\geq \gamma_1(A)$ and put
\[
\Delta \geq C_1 \frac{\gamma_1(A)}{m} \log^3\left(\frac{em}{\gamma_1(A)}\right) .
\]
Then with probability at least $1-2\exp(-c_2\Delta m)$,	 
	\begin{align}
	\label{eq:matrix.coordinate}
	\sup_{x\in A } \left( \frac{1}{m} \sum_{i=1}^m \left( (\Gamma x)_i^\sharp - \lambda_i \right)^2 \right)^{1/2}
	\leq C_3 \sqrt{\Delta \log\left(\frac{1}{\Delta}\right) }. 
	\end{align}
\end{theorem}
\end{tcolorbox}

The estimate in Theorem \ref{thm:wasserstein.matrix} is a combination of different behaviours in two regimes: for $i\leq \Delta m$ (and similarly for $i\geq (1-\Delta)m$), one expects no cancellations between the $(\Gamma x)^\sharp_i$s and the $\lambda_i$s.
As a result, the estimate in \eqref{eq:matrix.coordinate} follows from a standard argument (e.g.\ chaining or the strong-weak moment inequality for the gaussian random vector) that shows that for typical realizations of $(G_i)_{i=1}^m$, 
\[ \sup_{x\in A} \left( \frac{1}{m}\sum_{i=1}^{\Delta m} \left( (\Gamma x)^\sharp_i \right)^2 \right)^{1/2}
\leq C_1 \sqrt{\Delta \log\left(\frac{1}{\Delta}\right)}.\]

On the other hand, for $\Delta m \leq i \leq (1-\Delta)m$ one expects that cancellations do take place.
However, the best known bounds on the behaviour of $(\Gamma x)^\sharp$ that hold uniformly in $x$ take the form of `envelope estimates', namely that with probability at least $1-2\exp(-c_2\Delta m)$, for every $\Delta m \leq i \leq (1-\Delta) m$,
\begin{align}
\label{eq:envelope}
\sup_{x\in A}|(\Gamma x)^\sharp_i|
\leq C_3 \sqrt{2\log\left(\frac{m}{i}\right)}
\sim \left|F_g^{-1}\left(\frac{i}{m}\right) \right|.
\end{align}
There is a clear difference between \eqref{eq:envelope} and \eqref{eq:matrix.coordinate}:
The envelope estimate in \eqref{eq:envelope} is a  tail bound for $(\Gamma x)^\sharp$; while \eqref{eq:envelope} is a concentration bound.
It is important to stress that  \eqref{eq:envelope} does \emph{not} imply that
\begin{align}
\label{eq:matrix.coordinate.cancellation}
 \sup_{x\in A } \left( \frac{1}{m} \sum_{i=\Delta m}^{(1-\Delta)m} \left( (\Gamma x)_i^\sharp - F_g^{-1}\left(\frac{i}{m}\right) \right)^2 \right)^{1/2}
	\leq C_4 \sqrt{\Delta \log\left(\frac{1}{\Delta}\right) }.
\end{align}
Firstly,  at the very least one would need a matching uniform lower bound on $|(\Gamma x)_i^\sharp|$, and those are far more delicate than the uniform upper ones.
And secondly and more importantly, the estimate in \eqref{eq:matrix.coordinate.cancellation} shows that $(\Gamma x)^\sharp_i$ concentrates \emph{sharply} around $F_g^{-1}(\frac{i}{m})$, while the best that an envelope estimate can lead to is that 
\[ c F_g^{-1}\left(\frac{i}{m}\right)
\leq (\Gamma x)^\sharp_i
\leq  C F_g^{-1}\left(\frac{i}{m}\right)\] 
for constants $c$ and $C$ that need not be close to 1.

The known (envelope-)estimates on $(\Gamma x )^\sharp$  give enough information for addressing many nontrivial problems (for example, controlling the extremal singular values of random matrices with iid, heavy-tailed rows, see, e.g.\ \cite{adamczak2010quantitative,mendelson2012generic,mendelson2014singular,tikhomirov2018sample};  estimates on the quadratic empirical process, see, e.g.\ \cite{mendelson2010empirical}; or a functional Bernstein inequality for non-gaussian random matrices, see, e.g.\ \cite{bartl2022random}), but they do not capture the entire picture of the coordinate structure of $\Gamma (A)$ in the way that Theorem \ref{thm:wasserstein.matrix} does.

\begin{tcolorbox}
Theorem \ref{thm:wasserstein.matrix} is an outcome of  a more general phenomenon: 
As we explain in  Section \ref{sec:Wasserstein}, 
a uniform scale sensitive DKW inequality implies that for every $x$, $F_{m,x}$ and $F_g$ are close in the \emph{Wasserstein distance} $\mathcal{W}_2$.
\end{tcolorbox}

\subsection{A Sudakov-type bound}
\label{sec:Sudakov}

%

Let us show that the restriction on $\Delta$ in Theorem \ref{thm:intro.gaussian} cannot be significantly improved.
Ignoring logarithmic factors, Theorem \ref{thm:intro.gaussian} implies that with high probability, for every $t\in\R$ satisfying  that $\sigma^2(t)\geq \frac{\gamma_1(A)}{m}$, 
\begin{align}
\label{eq:intro.gaussian.for.lower.bound}
\sup_{x\in A} \left| \P_m(\inr{G,x} \leq t ) - \P(\inr{G,x} \leq t )\right| 
\lesssim \sigma(t) \sqrt{ \frac{\gamma_1(A)}{m}  }.
\end{align}

To demonstrate that \eqref{eq:intro.gaussian.for.lower.bound} is almost optimal, it actually suffices to focus on a single specific value for $t$, such as $t=0$ in which case $\sigma^2(t)=\frac{1}{4}$.
Let $\delta>0$ and denote by $\mathcal{N}(A,\delta B_2^d) $ the \emph{covering number} of $A$ at scale $\delta$ with respect to the Euclidean distance.
Thus, $\mathcal{N}(A,\delta B_2^d) $ is the smallest number of open Euclidean balls with radius $\delta$ needed to cover $A$.

\begin{tcolorbox}
\begin{proposition} \label{prop:intro.sudakov}

	There are absolute constants $C$ and $c$ such that the following holds.
	Let $A \subset S^{d-1}$ be a symmetric set and let $\delta>0$.
	\linebreak
	If  $m\geq C\max\{\frac{1}{\delta^2},\frac{\log \mathcal{N}(A,\delta B_2^d)}{\delta}\}$ then with probability at least $0.9$,
	\begin{align}	
	\label{eq:intro.sudakov}
	 \sup_{x\in A} \left| \P_m(\inr{G,x} \leq 0 ) - \P(\inr{G,x} \leq 0 )\right|
	\geq c \sqrt \frac{ \delta \cdot  \log \mathcal{N}\left(A,\delta B_2^d\right)  }{m} .
	\end{align}
\end{proposition}
\end{tcolorbox}

Proposition \ref{prop:intro.sudakov} is a Sudakov-type  bound that almost coincides with $\sqrt{ \gamma_1(A)/m }$: it is standard to verify that for any $A \subset \R^d$,
\begin{align}
\label{eq:packing.integral}
\sup_{\delta > 0 } \delta \cdot  \log \mathcal{N}\left(A,\delta B_2^d\right)
\leq
\gamma_1(A)
\leq
C \int_0^\infty  \log \mathcal{N}\left(A,\delta B_2^d\right)  \,d\delta.
\end{align}
The lower bound in \eqref{eq:packing.integral} is evident because for an optimal admissible sequence, $\sup_{x \in A} 2^s $ $\|x - \pi_s x\|_2 \leq \gamma_1(A)$, while the upper bound follows by the choice of suitable $\delta$-covers of $A$ as an admissible sequence.

In the case that interests us, when $A\subset S^{d-1}$ is a symmetric set, both inequalities are sharp up to a logarithmic factor in the dimension $d$ (see, e.g., \cite[Section  2.5]{talagrand2022upper}); hence
\[
\sqrt \frac{ \gamma_1(A)  }{m}
	\leq C \log(d) \sup_{\delta > 0 }  \sqrt \frac{ \delta \cdot  \log \mathcal{N}\left(A,\delta B_2^d\right) }{m} .
\]
The combination of these observations means that the restriction on $\Delta$ in Theorem \ref{thm:intro.gaussian} cannot be substantially improved and is sharp for a variety of sets $A$.
Moreover, it is sharp asymptotically (up to logarithmic factors):
\begin{align*}
&\liminf_{m\to\infty} \, \sqrt m \, \E\sup_{x\in A} \left| \P_m(\inr{G,x} \leq 0 ) - \P(\inr{G,x} \leq 0 )\right|
\\
&\qquad	\geq c \sup_{\delta > 0 } \sqrt{ \delta \cdot  \log \mathcal{N}\left(A,\delta B_2^d\right)  } 
\geq \frac{c' \sqrt{ \gamma_1(A)}}{\log(d)} .
\end{align*}

\begin{remark}
It is instructive to see the significant difference between the estimate of $\Delta \sim \frac{1}{m} \gamma_1(A)$ established in Theorem \ref{thm:intro.gaussian}, and the intuitive guess of $\Delta \sim \frac{1}{m} \gamma_2^2(A)$. 
For `large' subsets of $S^{d-1}$ (e.g., $S^{d-1}$ itself) $\gamma_2^2(A) \sim \gamma_1(A)$, but for smaller sets the gap between $\gamma_2^2(A)$ and $\gamma_1(A)$ is substantial. As an example, let $A = (e_1 + \frac{1}{\sqrt d} B_2^d) \cap S^{d-1}$, which is the spherical cap centred at $e_1$ of radius $1/\sqrt{d}$. It is straightforward to verify that $\gamma_2(A \cup -A)\sim 1$  while $\gamma_1(A \cup -A)\sim \sqrt d$.
\end{remark}

Finally, let us mention that  Proposition \ref{prop:intro.sudakov}  clearly implies that if $m\geq C d$  then with probability at least $0.9$, 
\begin{align}
\label{eq:for.vc}
\sup_{x\in S^{d-1}} \left| \P_m(\inr{G,x} \leq 0 ) - \P(\inr{G,x} \leq 0 )\right| 
\geq c\sqrt \frac{d}{m}.
\end{align}
On the other hand, for $m<d$, Remark \ref{rem:m.smaller.d} implies that with probability 1 the left hand side in \eqref{eq:for.vc} is at least $1/2$.
In particular the restriction in Theorem \ref{thm:intro.vc} that $\Delta \gtrsim\frac{d}{m}$ is necessary.

\subsection{More general random vectors}

Theorem \ref{thm:intro.gaussian} is close to a complete answer in the gaussian case, but it does not reveal what happens when $X$ is a more general isotropic random vector.
Typically, upper estimates on the difference between   true and empirical quantities that are established  for the gaussian random vector using chaining arguments are true for more general random vectors.
One might expect  a similar phenomenon  here, but as the next example shows, that is not the case.

\begin{example} \label{ex:intro.density}
There is an absolute constant $C$ and for every $m\geq C$ there is $d\in\N$ for which the following holds.
 If $X$ is uniformly distributed in $\{-1,1\}^d$, then  there is a set $A\subset S^{d-1}$ with $\gamma_1(A)\leq 1$ and 
\begin{align*}
\P\left( \sup_{x\in A} \sup_{t\in \R} \left| \P_m(\inr{X,x} \leq t ) - \P(\inr{X,x} \leq t )\right| \geq  \frac{1}{10} \right)\geq 0.9.
\end{align*}
\end{example}

In contrast, Theorem \ref{thm:intro.gaussian} guarantees that with high (polynomial) probability,
\[\sup_{x\in A} \sup_{t\in \R} | \P_m(\inr{G,x} \leq t ) - \P(\inr{G,x} \leq t )| 
\leq C \sqrt \frac{\log^3(m)}{ m}.\]

\begin{remark}
The set $A$ used in Example \ref{ex:intro.density} can be chosen to have a Euclidean  diameter that is arbitrarily small.
Thus, a subgaussian assumption alone cannot yield a satisfactory estimate on $\sup_{x\in A} \sup_{t\in \R} | \P_m(\inr{X,x} \leq t ) - \P(\inr{X,x} \leq t )|$ that is based on \emph{any} reasonable complexity parameter of $A$.
\end{remark}

The behaviour of $| \P_m(\inr{X,x} \leq t ) - \P(\inr{X,x} \leq t )|$, and specifically a uniform version of \eqref{eq:dkw.thm} for more general random vectors is studied in \cite{bartl2023dvoretsky},  where we show that even seemingly nice random vectors do not satisfy Theorem \ref{thm:intro.gaussian}.

Finally, let us comment on some related literature concerning estimates on the normalized difference between empirical and true distribution  that hold uniformly over a class of functions.
The article closest to the present one is \cite{lugosi2020multivariate} where an estimate as in Theorem \ref{thm:intro.gaussian} is proven, but with a different restriction on $\Delta$.
That restriction is formulated via a fixed point condition (involving the $L_2(\P)$-covering numbers and a Rademacher complexity) and it not optimal; for example, one can easily construct sets $A$ where the restriction on $\Delta$ in \cite{lugosi2020multivariate} requires that $\Delta \gtrsim (\frac{\gamma_1(A)}{m} )^{2/5}$.
Previous estimates, such as those presented in \cite{alexander1987central,alexander1987rates,gine2006concentration,gine2003ratio,koltchinskii2003bounds,koltchinskii2002empirical},
where formulated either in terms of conditions on the random metric structure of the function class (respectively of their sub-level sets), or in terms of VC-conditions, and are not strong enough to be applicable in the present setting.

\vspace{0.5em}
We end the introduction with a word about \emph{notation}.
Throughout, $c,c_0,c_1,C,C_0,C_1\dots$ denote strictly positive absolute constants that may change their value in each appearance.
We make the convention that $c$ are small constants with values in $(0,1)$, and that $C$ are large constants with values in $[1,\infty)$.
If a constant $c$ depends on a parameter $\kappa$, that is denoted by $c=c(\kappa)$.
We write $\alpha \lesssim \beta$ if there is an absolute constant $C$ such that $\alpha\leq C \beta$ and $\alpha \sim \beta$ if both $\alpha \lesssim \beta $ and $\beta \lesssim \alpha$.

\section{Preliminaries}

In this section we collect several simple estimates that play a crucial role in the proof of Theorem \ref{thm:intro.gaussian}, and we assume that $\Delta\leq\frac{1}{10}$.

We start by  `discretizing $\R$' and show that if $|F_{m,x}(t) - F_g(t)|$ is small for every $t\in \R$ that belongs to a suitable grid, then  $|F_{m,x}(t) - F_g(t)|$ is small for every $t\in\R$.
To that end, let
\[U_\Delta=\left\{ \ell\Delta : \ell\in\N, \, 1\leq \ell \leq \frac{1-\Delta}{\Delta } \right\}
\subset[\Delta,1-\Delta]\]
and set 
\[R_\Delta = \left\{t\in\R :  F_g(t)\in U_\Delta \right\}.\]
Recall that $\sigma^2(t)=F_g(t)(1-F_g(t))\in[0,\frac{1}{4}]$.

\begin{lemma}
\label{lem:grid}
	Let $A'\subset S^{d-1}$, $\alpha\geq 0$ and  $\beta\geq 1$.
	Fix a realization  of $(G_i)_{i=1}^m$ and assume that 	for every $x\in  A'$ and $t\in R_\Delta$,
	\begin{align}
	\label{eq:grid.assumption}
	|F_{m,x}(t) - F_g(t) | 
	\leq \alpha+\beta \sigma(t) \sqrt{\Delta}. 
	\end{align}
	Then for every $x\in A'$ and $t\in \R$,
	\[ |F_{m,x}(t) - F_g(t) | 
	\leq  \alpha+ \beta\sigma(t)\sqrt \Delta + (\beta+1)\Delta . \]
\end{lemma}
\begin{proof}
	Let $t\in\R$ and  consider the case $F_g(t)\in[\Delta,1-\Delta]$.
	Let $s$ be the smallest element in $R_\Delta$ with $s\geq t $; hence, $|F_g(t)-F_g(s)|\leq \Delta$ and $\sigma^2(s) \leq \sigma^2(t) + \Delta$.
	It follows from the monotonicity of $F_{m,x}$ and \eqref{eq:grid.assumption} that
	\begin{align*}
	F_{m,x}(t) 
	\leq F_{m,x}(s)
	&\leq F_{g}(s) +  \alpha+\beta \sqrt{\sigma^2(s)\Delta  }
	\\
	&\leq F_g(t) +  \Delta +  \alpha+\beta \sqrt{\sigma^2(t) +\Delta} \sqrt{\Delta}.
	\end{align*}
	Therefore, using the  sub-additivity of $u\mapsto\sqrt u$, 
	\[F_{m,x}(t) - F_g(t)
	\leq \alpha + (1+\beta)\Delta + \beta  \sigma(t)\sqrt{\Delta} .\]
	The corresponding lower bound follows from the same argument and is omitted.
	
	Next, assume that $F_g(t)\notin[\Delta,1-\Delta]$ and  consider the case $F_g(t)\leq \Delta$ (the  case where $F_g(t)\geq 1-\Delta$ is identical).
	Let $s\in R_\Delta$ satisfy that $F_g(s)=\Delta$, and in particular $s\geq t$.
	By the monotonicity of $F_{m,x}$,
	\[ F_{m,x}(t)
	\leq F_{m,x}(s)
	\leq F_g(s) +\alpha + \beta\sigma(s)\sqrt{\Delta}
	\leq \alpha+ (1+\beta)\Delta, \]
	and therefore
	\[| F_{m,x}(t) - F_g(t)|
	\leq \max\{ F_{m,x}(t) , F_g(t)\}
	\leq \alpha+(1+\beta)\Delta.
	\qedhere \]
\end{proof}

The next observation we require is that $x\mapsto |F_{m,x}(t)- F_g(t)|$ is continuous in an appropriate sense.

\begin{lemma}
\label{lem:end.chain.to.everything}
	For every $x,y\in S^{d-1}$, $t\in\R$, and $\xi>0$,
	\begin{align*}
	|F_{m,x}(t) - F_g(t) |
	&\leq \sup_{t'\in[t-\xi,t+\xi]} |F_{m,y}( t') - F_g(t') | + \P_m( |\inr{G,x} - \inr{G,y} | \geq \xi) \\
	&\qquad  + \left( F_g(t+\xi)-F_g(t-\xi) \right). 
	\end{align*}
\end{lemma}
\begin{proof}
	Set $X=\inr{G,x}$ and $Y=\inr{G,y}$.
	Then
	\begin{align*}
	F_{m,x}(t)
	=\P_m(X\leq t)
	&=\P_m( X\leq t , |X-Y|< \xi) + \P_m(X\leq t , |X-Y|\geq \xi)\\
	&\leq \P_m( Y\leq t+\xi) + \P_m (|X-Y|\geq \xi) .
	\end{align*}
	Moreover, 
	\begin{align*}
	\P_m( Y\leq t+\xi)
	=F_{m,y}(t+\xi)
	&\leq F_g(t+\xi) + |F_{m,y}(t+\xi)-F_g(t+\xi)|\\
	& = F_g(t) + (F_g(t+\xi)-F_g(t)) + |F_{m,y}(t+\xi)-F_g(t+\xi)|.
		\end{align*}
	The claimed upper estimate on $F_{m,x}-F_g$ is evident by combining these observations, while the lower one follows from an identical argument and is omitted.
\end{proof}

With Lemma \ref{lem:end.chain.to.everything} in mind, it is natural to explore the behaviour of  $F_g(t+\xi)-F_g(t-\xi)$.
For $u\in[0,1]$ set
\[ \bar{u}=\min\{u,1-u\} 
\]
and note that $\bar{u}\sim u (1-u)$; in particular $\sigma^2(t)\sim \overline{F_g(t)}$.

In what follows we use four features of the gaussian random vector: rotation invariance; the tail-decay of one-dimensional  marginals, i.e., that $\P(|g|\geq t)\leq 2\exp(-t^2/2)$; the fact that the gaussian density $f_g$ satisfies
\begin{align}
\label{eq:cheeger}
f_g\left( F_g^{-1}(u) \right)
\leq C_0 \bar{u} \sqrt{ \log\left( \frac{1}{ \bar{u} } \right) },
\end{align} 
see e.g., \cite[Lemma 5.2]{csiszar2011information};
and the two sided bound on the gaussian distribution function (see, e.g., \cite[Lemma 2, page 175]{feller1991introduction}) which immediately implies that for all $t\geq 1$
\begin{align}
\label{eq:sigma.equivalent}
\sigma^2(t) 
\sim \frac{1}{t}\exp\left(-\frac{t^2}{2}\right).
\end{align}
Let us mention two immediate consequences of \eqref{eq:sigma.equivalent} that are used in what follows.

\begin{lemma}
\label{lem:estimate.sigma}
	There are absolute constants $c$ and $C$ such that for every $t\geq 1$, we have that
	$\log(\frac{1}{\sigma^2(t)} )	\in [c t^2, C t^2]$ and for  $\eta\in(1,\frac{1}{t^2})$,  $\sigma^2(t + \eta t) \in [c\sigma^2(t),\, C \sigma^2(t) ]$.
\end{lemma}
\begin{proof}
	The first claim is a direct consequence of \eqref{eq:sigma.equivalent}.
	For the second claim, note that $\eta t\leq 1$ and thus $t+\eta t\sim t$.
	Moreover, 
	$  (t+\eta t)^2 - t^2 
	=  2\eta t^2 + \eta^2 t^2
	\leq 3$
	and thus $\exp(-t^2/2)\sim \exp(-(t+\eta t)^2/2)$, therefore the second claim also follows from  \eqref{eq:sigma.equivalent}.
\end{proof}

\begin{remark}
The bound in \eqref{eq:cheeger}  is actually two-sided.
The lower bound implies Cheeger's isoperimetric inequality (see, e.g.\ \cite{bobkov1997isoperimetric,kannan1995isoperimetric} for more details).
\end{remark}

Clearly, \eqref{eq:cheeger} is equivalent to
\[f_g(t)\leq C \sigma^2(t) \sqrt{  \log\left(\frac{1}{\sigma^2(t)}\right) }  \] 
for a suitable absolute constant $C$, which suggests that for small $\xi>0$, 
\begin{align}
\label{eq:derivative}
F_g(t+\xi)-F_g(t) 
\approx \xi f_g(t) \leq C \xi \sigma^2(t)  \sqrt{ \log\left(\frac{1}{\sigma^2(t)}\right) }. 
\end{align}
The next straightforward (but slightly technical) observation makes \eqref{eq:derivative} precise.

\begin{lemma}
\label{lem:regularity}
	There are absolute constants $c$ and $C$ such that the following holds.
	For every $t\in\R$ and  $0\leq \xi\leq  \frac{c}{ \sqrt{ \log(1/\sigma^2(t)) }}$, 
	\[ |F_g(t \pm \xi ) - F_g(t) | \leq   C \xi \sigma^2(t) \sqrt{ \log\left(\frac{1}{ \sigma^2(t) } \right) }.\]
\end{lemma}
\begin{proof}
	Set $u=F_g(t)$.
	We only present the case  $u\leq \frac{1}{2}$ (and in particular $u\sim\sigma^2(t)$); the details when $u>\frac{1}{2}$ are left to the reader.
	
	Let $C_0$ be the constant appearing in \eqref{eq:cheeger} and set 
	\[\beta=2C_0 \xi u \sqrt{ \log\left(\frac{1}{u}\right)}.\]
	To prove that  $F_g(t+\xi)\leq F_g(t)+\beta$  it suffices to  show that $t+\xi \leq F_g^{-1}(u+\beta)$.
	By the fundamental theorem of calculus and \eqref{eq:cheeger},
	\begin{align}
	\label{eq:inv.cheeger.in.proof}
	F_g^{-1}(u+\beta) -F_g^{-1}(u)
	&=\int_u^{u+\beta} \frac{dv}{f_g(F_g^{-1}(v))} 
	\geq \frac{1}{C_0} \int_u^{u+\beta} \frac{dv}{ \bar{v} \sqrt{ \log(1 / \bar{v} ) } }.
	\end{align}
	It is straightforward to verify that for $v\in[u,2u]$, 
	$\bar{v} \sqrt{ \log(1 / \bar{v} )}\leq 2 u \sqrt{ \log(1 / u )}.$
	Therefore, if  $\xi \leq \frac{1}{2C_0 \sqrt{ \log(1/u)} }$ then $\beta\leq u$ and \eqref{eq:inv.cheeger.in.proof} implies that
	\[ F_g^{-1}(u+\beta) -F_g^{-1}(u)
	\geq \frac{\beta }{ 2C_0 u \sqrt{ \log(1/u) }}
	= \xi,
 \]
	by the definition of $\beta$.
	Since $t=F_g^{-1}(u)$, this proves that  $t+\xi \leq F_g^{-1}(u+\beta)$, as claimed.
	
	The proof that $F_g(t-\xi)\geq F_g(t)-\beta$ follows a similar (in fact---simpler) path.
\end{proof}

The next lemma is the key to the proof of Theorem \ref{thm:intro.gaussian}---which is based on a (non-standard) chaining argument carried out in a class of indicator functions.
For such a chaining argument one has to control increments of the form 
\[ \left\| 1_{(-\infty,t]}(\inr{G,x}) - 1_{(-\infty,t]}(\inr{G,y}) \right\|_{L_2}
=\P^{1/2}\big(\left\{ \inr{G,x}\leq t \right\} \,\,\Delta \,\, \left\{ \inr{G,y}\leq t\right\} \big) ,\]
where $A \, \Delta \, B$ is the symmetric difference between the two sets.
The wanted estimate follows from the next lemma.

\begin{lemma}
\label{lem:sym.diff}
	There is an absolute constant $C$ such that for every $x,y\in S^{d-1}$ and $t\in\R$, setting $\delta=\|x-y\|_2$, we have that
	\begin{align}
	\label{eq:symm.diff}
	\P\big(\left\{ \inr{G,x}\leq t \right\} \,\,\Delta \,\, \left\{ \inr{G,y}\leq t\right\} \big) 
	\leq  C \sigma^2(t) \delta\log\left(\frac{1}{\sigma^2(t) \delta} \right).
	\end{align}
\end{lemma}
\begin{proof}
	\emph{Step 1:}
	Set
	\[ \mathcal{D}= \left\{ z\in \R^d : \inr{z,x}\leq t \right\} \,\,\Delta \,\, \left\{ z\in \R^d  :\inr{z,y}\leq t\right\}, \]
	and the left hand side in \eqref{eq:symm.diff} is $\P(G\in\mathcal{D})$.
	By symmetry and  the rotation invariance of the gaussian measure we may assume that $t\geq 0$, that $x=e_1$  and $y\in{\rm span}\{e_1,e_2\}$ and  that $d=2$.
	Also, note that if $x$ and $y$ are `far away' from each other, \eqref{eq:symm.diff} is trivially true.
	Indeed, since $t\geq 0$ and thus $\sigma^2(t) \geq \frac{1}{2} (1-F_g(t))$, by the union bound 
	\[ \P(G\in\mathcal{D})
	\leq \P( \inr{G,y}>t ) + \P(\inr{G,x}>t)
	=2( 1-F_g(t))
	\leq 4 \sigma^2(t) . \]
	Therefore it suffices to consider  $\delta=\|x-y\|_2$ for which $\delta\log(\frac{1}{ \sigma^2(t) \delta})\leq c_0$ for a suitable absolute constant $c_0$; in particular we may assume that $\delta\leq \frac{1}{10}$.

	\vspace{0.5em}
	\noindent
	\emph{Step 2:}
	We first focus on the more involved case when $t\geq 1$. 
	Denote by $\alpha$ the angle between $x$ and $y$ and note that $\alpha\sim \delta$.
	Observe that $\mathcal{D}$ is the union of two cones with angle $\alpha$,  see Figure \ref{fig:cone}.
\begin{figure}[h]
\includegraphics[scale=0.13]{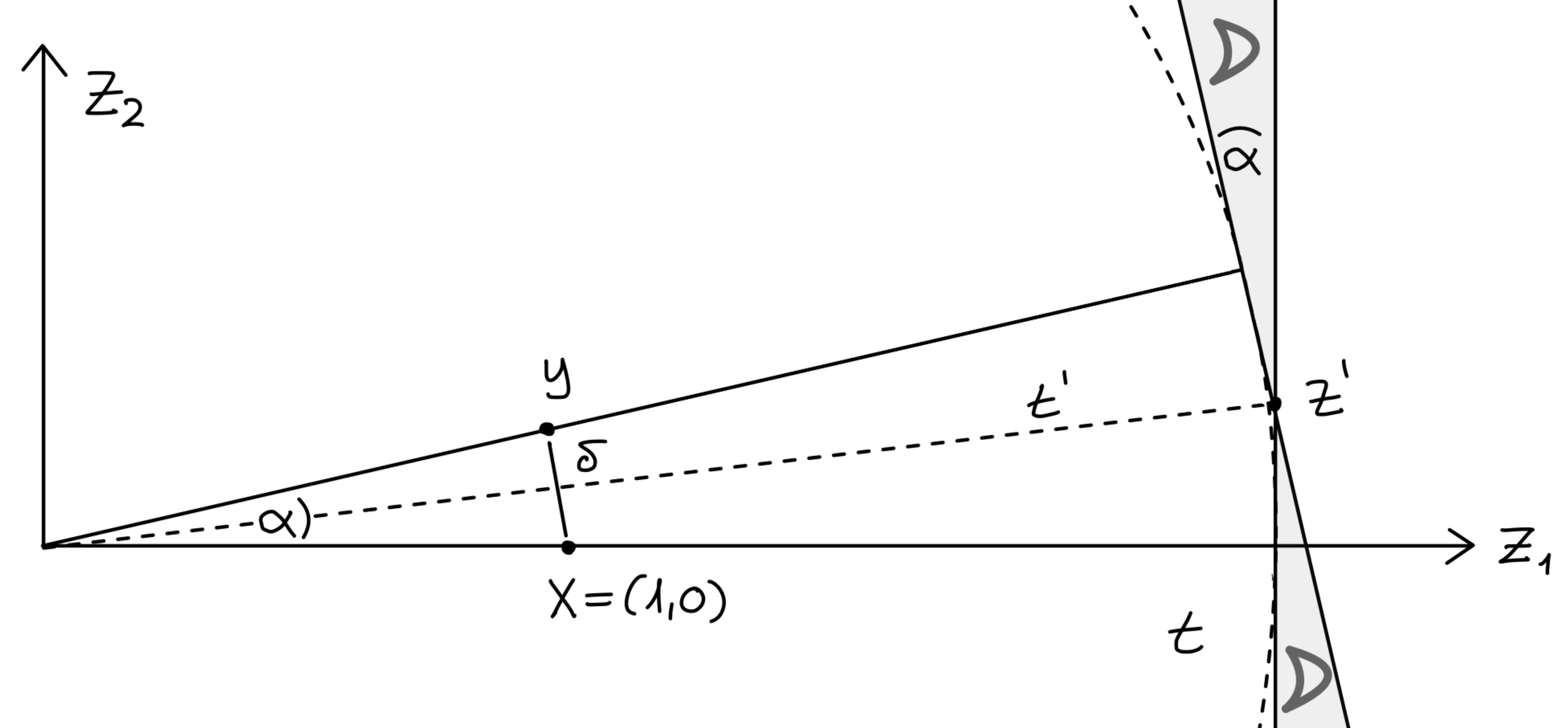}
\vspace*{-2em}
\hspace*{-3em}
\caption{The set $\mathcal{D}$.}
\label{fig:cone}
\end{figure}
	Moreover, also observe that $z'$ (the intersection of the lines $\inr{\,  \cdot  ,  x }=t $ and $\inr{ \,\cdot  ,y}=t$)  satisfies   $t'=\|z'\|_2\in[t, t(1+C_1\delta^2)]$.
	To estimate $\P(G\in\mathcal{D})$, rotate $\mathcal{D}$ by $\frac{\alpha}{2}$  mapping  $z'$ to the vector $(t',0)$.
	Writing $G=(g_1,g_2)$ and invoking rotation invariance, we have that
	\begin{align}
	\nonumber 
	 \P(G\in \mathcal{D})
	&\leq \P\left( g_1 \in \left[ t'- C_2\delta |g_2|,  t' + C_2\delta |g_2| \right] \right) \\
	\label{eq:estimate.D}
	&= \E_{g_2} \left( F_g\left( t' + C_2\delta |g_2| \right) - F_g\left( t' - C_2\delta |g_2| \right)\right).
	\end{align}
	
	Set $\beta =  \sqrt{ 2\log(\frac{1}{ \sigma^2(t)\delta}) }$	and note that  $\beta\geq 1$ because  $\delta\leq \frac{1}{10}$ and $\sigma^2(\cdot)\leq \frac{1}{4}$.
	By considering the sets  $\{|g_2|\leq \beta\}$ and $\{|g_2|> \beta\}$, 
	it is evident from \eqref{eq:estimate.D} that 
	\[\P(G\in \mathcal{D})
	\leq\left(F_g\left( t' + C_2\delta \beta \right) - F_g\left( t' - C_2\delta \beta \right) \right) + \P(|g_2|> \beta ).
	\]
	Clearly
	\[\P(|g_2|> \beta)
	\leq 2\exp\left(-\frac{\beta^2}{2}\right)
	\leq  2\sigma^2(t)\delta, \]
	and to control the term $F_g( t' + C_2 \delta \beta) - F_g(t'-C_2\delta\beta)$ we shall use Lemma \ref{lem:regularity} (with the choices $t'$ and $\xi=C_2\delta\beta$).
	To that end, first observe that $\sigma^2(t')\sim\sigma^2(t)$.
	Indeed,  $|t'-t|\leq C_1 \delta^2 t$  and  $t^2 \leq C_4\log(\frac{1}{\sigma^2(t)})$ by Lemma \ref{lem:estimate.sigma}, and the restriction on $\delta$ implies that  $C_1\delta^2 \leq \frac{1}{t^2}$; thus the claim follows by another application of Lemma \ref{lem:estimate.sigma}.
	
	Next observe that $\xi=C_2\delta\beta$ is a admissible choice in Lemma \ref{lem:regularity}.
	Indeed, using that $\sigma^2(t')\sim\sigma^2(t)$ and the  restriction on $\delta$,
	\[\xi = C_2\beta \delta 
	= C_2\delta \sqrt{ 2 \log\left( \frac{1}{ \sigma^2(t) \delta } \right)}
	\leq \frac{ C_5c_0}{\sqrt{ \log(1/\sigma^2(t')) }} ,\]
	and $c_0$ may be chosen to be as small as needed.
	Therefore, by Lemma \ref{lem:regularity},
	\begin{align*}
	F_g\left( t' + C_2 \delta \beta \right) - F_g\left( t' - C_2 \delta \beta \right)
	&\leq C_6  \delta \beta \sigma^2(t' ) \sqrt{ \log\left(\frac{1}{\sigma^2(t')}\right) } \\
	&\leq  C_7 \delta \sigma^2(t) \log\left( \frac{1}{\sigma^2(t) \delta } \right),
	\end{align*}
	where we again used that $\sigma^2(t')\sim \sigma^2(t)$.
	This concludes the proof of the wanted estimate on $\P(G\in\mathcal{D})$ in case $t\geq 1$.
	
	\vspace{0.5em}
	\noindent
	\emph{Step 3:}
	We are left with the case $t\in[0,1)$.
	Clearly  $\sigma^2(t),\sigma^2(t)\geq c_8$ and therefore it suffices to show that  $\P(G\in\mathcal{D})\leq C \delta \log(\frac{1}{\sigma^2(t)\delta})$.
	To that end, we follow the same arguments as in step 2, but this time we simply estimate
	\[ F_g(t'+C_2 \delta \beta) -  F_g(t'-C_2\delta \beta)
	\leq 2C_2\delta \beta  \]
	since the gaussian density $f_g$ satisfies $f_g(\cdot)\leq 1$.
	Thus the claim follows from the definition of $\beta$.	
\end{proof}

\section{Proof of Theorem \ref{thm:intro.gaussian}}
\label{sec:proof.ratio}

As noted previously, the proof of Theorem \ref{thm:intro.gaussian} is based on a chaining argument.
Without loss of generality  let $\Delta\leq\frac{1}{10}$ and $m\geq 10$.
Let $C_0$ be a (large) absolute constant that is specified in what follows and assume that
\begin{align}
\label{eq:cond.Delta}
 \frac{ \Delta m }{ \log^3(e /\Delta )}
\geq C_0 \gamma_1(A) .
\end{align}
In particular, since  $\gamma_1(A)\geq 1$ by the symmetry of $A$, we have that $\Delta m\geq C_0\geq 1$.
Set $(A_s)_{s\geq 0}$  to be an almost optimal admissible sequence for $\gamma_1(A)$, i.e.,
\begin{align}
\label{eq:optimal.admissible.seq}
\sup_{x\in A} \sum_{s\geq 0 } 2^{s}  \| x-\pi_s x \|_{2}
\leq 2 \gamma_1(A).
\end{align}
Recall that $\bar{u}=\min\{u,1-u\}$, and for $u\in[\Delta,1-\Delta]$ let  $1\leq s_0\leq r_u$ be the smallest integers that satisfy
\[ 2^{s_0} \geq \Delta m
 \quad \text{and}\quad
2^{r_u} \geq \sqrt{\bar{u} \Delta} m ; \]
in particular $2^{s_0}\leq 2\Delta m$ and $2^{r_u}\leq 2\sqrt{ \bar{u} \Delta }m$ since $\Delta m\geq 1$.

\vspace{1em}
The proof of Theorem \ref{thm:intro.gaussian} consists of three steps that are outlined here.
The complete arguments are presented in Sections \ref{sec:down.chain}, \ref{sec:control.large.coord}, and \ref{sec:putting.together}.

Using the notation of Lemma \ref{lem:end.chain.to.everything}, set $u\in[\Delta,1-\Delta]$ and let $y=\pi_{r_u} x$.
By that lemma, for $\xi>0$ we have that
\begin{align*}
|F_{m,x}(t) - F_g(t) |
&\leq \sup_{t'\in[ t-\xi,t+\xi]} |F_{m,\pi_{r_u}x}(t') - F_g(t') | + \P_m( |\inr{G,x} - \inr{G,\pi_{r_u}x} | \geq \xi) \\
	&\qquad  +\left( F_g(t+\xi)-F_g(t-\xi) \right)
\\
	&=(1)+(2)+(3),
\end{align*}
and the goal is to show that all three terms are (at most) of order $\sigma(t)\sqrt\Delta$.

The requirement that $(3)\lesssim \sigma(t)\sqrt\Delta$  dictates the correct choice of $\xi$: indeed, by  Lemma \ref{lem:regularity} (and ignoring logarithmic factors for the moment), if $\xi\sim \sqrt\frac{\Delta}{\sigma^2(t)}$, then
 \[F_g(t+\xi)-F_g(t-\xi)\sim \xi \sigma^2(t)\sim \sigma(t)\sqrt\Delta.\]
As for $(2)$, if $u=F_g(t)$ and in particular $\sigma^2(t)\sim \bar{u}$, it follows from the fact that $(A_s)_{s\geq 0}$ is an almost optimal admissible sequence  and   \eqref{eq:cond.Delta} that 
\[\|x-\pi_{r_u}x\|_2\leq 2 \frac{\gamma_1(A) }{ 2^{r_u} }
\lesssim \frac{\xi}{\log^3(1/\Delta)}.\]
Therefore, by the gaussian tail-decay, $\P( |\inr{G,x} - \inr{G,\pi_{r_u}x} | \geq \xi)$ is sufficiently small and the key is to show that the empirical counterpart, $\P_m( |\inr{G,x} - \inr{G,\pi_{r_u}x} | \geq \xi)$,  is small as well.

Finally, the heart of the proof is to establish a suitable bound on
\begin{align}
\label{eq:chaining.explain}
|F_{m,\pi_{r_u}x}(t') - F_g(t')|. 
\end{align}
The estimate on \eqref{eq:chaining.explain} relies on a somewhat unorthodox chaining argument, though its starting point  is rather obvious: seeing `how far' the scale sensitive DKW inequality for a single random variable takes us.

\begin{lemma}
\label{lem:single.function}
	There is an  absolute constant $C$ such that with probability at least $1-2\exp(-\Delta m)$, for every $x\in A$ and $t\in\R$,
	\begin{align}
	\label{eq:ratio.for.s0}
	\left| F_{m,\pi_{s_0}x}( t )-  F_{g}( t ) \right|
	\leq C \left( \Delta + \sigma(t)\sqrt{  \Delta  } \right).
	\end{align}
\end{lemma}

We only sketch the standard proof of Lemma \ref{lem:single.function}.
By Theorem \ref{thm:intro.ratio.single}, for any fixed $x\in A$, with probability at least $1-2\exp(-3\Delta m)$,  \eqref{eq:ratio.for.s0} holds for every $t \in \R$.
With the choice of $s_0$,
	\[|\{\pi_{s_0} x : x\in A \}|
	\leq 2^{ 2^{s_0} }
	\leq \exp( 2\Delta m), \]
and the claim is an immediate outcome of the union bound.

\subsection{Going further down the chain}
\label{sec:down.chain}

With $\pi_{s_0}x$ being the starting points of each chain, the next step is to go further down the chains---up to $\pi_{r_u}x$.
To that end, recall that 
\[ U_\Delta=\left\{\ell \Delta : \ell\in \N, \,1\leq \ell \leq \frac{1-\Delta}{\Delta} \right\} \subset [\Delta,1-\Delta].\]

\begin{lemma}
\label{lem:ratio.pi.r.u}
	There are absolute constants $c$ and  $C$  such that with probability at least $1-2\exp(-c\Delta m)$,  for every $u\in U_\Delta$, $t\in \R$ and $x\in A$,
	\begin{align}
	\label{eq:chaining.first.part}
	& \left| F_{m,\pi_{r_u}x} \left( t \right) - F_g(t) \right|
	\leq C\left(  \sqrt{\bar{u} \Delta } + \sigma(t) \sqrt{\Delta } \right).
	\end{align}
\end{lemma}

\begin{remark}
	Lemma \ref{lem:ratio.pi.r.u} will be applied with $u\sim F_g(t)$.
	Hence $\bar{u}\sim \sigma^2(t)$, and the right hand side of \eqref{eq:chaining.first.part} is proportional to $\sigma(t)\sqrt\Delta$.
\end{remark}

Before diving into technicalities, note that by  Lemma \ref{lem:grid} we may consider only 
\[t\in R_\Delta = \left\{ t\in \R : F_g(t) \in U_\Delta \right\}\]
(rather than $t\in \R$) in Lemma \ref{lem:ratio.pi.r.u}.
Moreover, since  $|U_\Delta|, |R_\Delta|\leq \frac{1}{\Delta}$ and  $ \Delta m \geq C_0 \log(\frac{1}{\Delta})$ for a constant $C_0$ that we are free to choose as large as we require, it is enough to show that for every \emph{fixed} $u\in U_\Delta$ and $t\in R_\Delta$, with probability at least $1-2\exp(-c\Delta m)$, \eqref{eq:chaining.first.part} holds uniformly for $x\in A$.
Indeed, once that is established, the wanted estimate  follows from the union bound.

\vspace{0.5em}
Fix $t\in R_\Delta$ and $u\in U_\Delta$.
Clearly
\begin{align*}
	\left| F_{m,\pi_{r_u}x} (t) - F_g(t) \right|
	&\leq \left| F_{m,\pi_{r_u}x} ( t) - F_{m,\pi_{s_0}x} ( t)\right|
	+ \left| F_{m,\pi_{s_0}x} ( t) - F_g(t)\right| \\
	&=(1) + (2),
\end{align*}
and by Lemma \ref{lem:single.function}, with high probability and uniformly in $x\in A$, 
\[ (2)\leq C\left( \Delta + \sigma(t)\sqrt{\Delta} \right).\]
To control $(1)$ we first use a larger auxiliary functional and then show that the functional is dominated by $\gamma_1$.

From here on, set 
\[ {\rm S}_t[x, y]
= \P\big( \left\{ \inr{G,x} \leq t\right\} \ \bigtriangleup \ \left\{ \inr{G,y}\leq t\right\}\big). \]

\begin{lemma}
\label{lem:chaning}
	There are absolute constants $c$ and $C$ such that for every  $r> s_0$ and $t\in R_\Delta$ the following holds.
	With  probability at least $1-2\exp(-c\Delta m)$, for every $x\in A$,
	\begin{align*}
	 &\left| F_{m,\pi_r x} (t) - F_{m,\pi_{s_0}x} ( t) \right|
	\leq C \sup_{x\in A} \sum_{s=s_0}^{r-1} \left( \frac{2^s}{m} + \sqrt{ \frac{2^s}{m} \P\left( {\rm S}_t[\pi_{s+1}x,\pi_s x] \right) } \right).
	\end{align*}
\end{lemma}


The proof of Lemma \ref{lem:chaning} relies on the following straightforward fact.

\begin{lemma}
\label{lem:chaning.basic.inequality}
	There is an absolute constant $C$ such that for every $x,y\in S^{d-1}$, $t\in\R$ and $\lambda \geq 0$, with probability at least $1-2\exp(- \lambda)$,
	\begin{align*}
	 \left| F_{m,x} \left( t\right) - F_{m,y} \left( t \right) \right|
	\leq C \left( \frac{ \lambda }{m} + \sqrt{\frac{ \lambda }{m} \P\left( {\rm S}_t[x,y] \right) } \right).
	\end{align*}
\end{lemma}
\begin{proof}
For $1\leq i \leq m$, set
\[Z_i=1_{(-\infty, t ]}(\inr{G_i,x}) -1_{(-\infty, t ]}(\inr{G_i,y})\]
and let $Z$ be distributed as $Z_1$.
Using this notation,
\begin{align*}
F_{m,x} \left( t \right) - F_{m,y} \left(t\right)
&=\frac{1}{m}\sum_{i=1}^m Z_i.
\end{align*}
Clearly $\|Z\|_{L_\infty}\leq 1$,  $\E Z^2=\P({\rm S}_t[x,y])$, and by the rotation invariance of the gaussian measure, $\E Z=0$. 
The claim follows immediately from Bernstein's inequality (see, e.g., \cite[Theorem 2.10]{boucheron2013concentration}).
\end{proof}

\begin{proof}[Proof of  Lemma \ref{lem:chaning}]
Fix  $t\in R_\Delta$ and $r> s_0$.
For every  $x\in A$, we have that
\[F_{m,\pi_{r} x} (t) - F_{m,\pi_{s_0 } x} ( t)
=\sum_{s=s_0}^{r-1}\left( F_{m,\pi_{s+1}x} (t ) - F_{m,\pi_{s}x} ( t )  \right).\]
By Lemma \ref{lem:chaning.basic.inequality}, with probability at least $1-2\exp(- 2^{s+3})$,
\begin{align}
\label{eq:chaining.1.step}
\left| F_{m,\pi_{s+1}x} (t ) - F_{m,\pi_{s}x} ( t )  \right|
\leq C\left( \frac{2^s}{m} + \sqrt{ \frac{2^s}{m} \P({\rm S}_t[\pi_{s+1} x,\pi_s x]) } \right).
\end{align}
Since
\[|\{(\pi_{s+1} x, \pi_s x) : x\in A \}|
\leq 2^{2^s}\cdot 2^{2^{s+1}}
\leq \exp\left( 2^{s+2} \right),
\]
it follows from the union bound over pairs $(\pi_{s+1} x,\pi_s x)$ that  with probability at least $1-2\exp(-2^{s+2})$, \eqref{eq:chaining.1.step}  holds uniformly for $x\in A$.
And, by the union bound over $s\in\{s_0,\dots,r-1\}$, with probability at least
\[ 1-\sum_{s=s_0}^{r-1} 2\exp(- 2^{s+2})
\geq 1- 2\exp(- 2^{s_0})
\geq 1-2\exp(-\Delta m), \]
we have that
\begin{align*}
\left| F_{m,\pi_{s+1}x} (t ) - F_{m,\pi_{s}x} ( t ) \right|
 &\leq  C \sup_{ x\in A } \sum_{s=s_0}^{r-1} \left( \frac{2^s}{m} + \sqrt{ \frac{2^s}{m} \P({\rm S}_t[\pi_{s+1} x, \pi_s x]) } \right) .
\qedhere
\end{align*}
\end{proof}

\begin{proof}[Proof of Lemma \ref{lem:ratio.pi.r.u}]
Fix $u\in U_\Delta$ and $t\in R_\Delta$; in particular $\sigma^2(t)\geq\frac{1}{2}\Delta$.
By Lemma \ref{lem:chaning}, it suffices to show that 
\begin{align*}
(\ast)=\sup_{ x\in A } \sum_{s=s_0}^{r_u-1} \left( \frac{2^s}{m} + \sqrt{ \frac{2^s}{m} \P({\rm S}_t[\pi_{s+1} x,\pi_s x]) } \right) 
\leq C_1\left( \sqrt{\bar{u} \Delta} + \sigma(t)\sqrt{\Delta} \right).
\end{align*} 
To that end, set $\delta_s(x)=\|\pi_{s+1}x-\pi_s x\|_2$.
Invoking Lemma  \ref{lem:sym.diff},
\[ \P\left({\rm S}_t\left[\pi_{s+1} x,\pi_s x\right]\right)
\leq C_2\sigma^2(t) \delta_s(x) \log\left(\frac{1}{\sigma^2(t) \delta_s(x) } \right).\]
Since $2^{s_0}\geq \Delta m$ and $2^{r_u}\leq 2 \sqrt{\bar{u}\Delta  } m \leq 2\sqrt{\Delta } m$, it is evident that 
\[ r_u -s_0 
\leq  \log_2(2\sqrt\Delta m ) - \log_2(\Delta m) 
\leq \log\left(\frac{4}{\Delta}\right).\]
Thus, 
\[  (\ast) \leq   \frac{2^{r_u}}{m} +  \sup_{x\in A} \max_{s_0\leq s<r_u} \log\left(\frac{4}{\Delta}\right) \sqrt{ \frac{ 2^s}{m} C_2 \sigma^2(t) \delta_s(x) \log\left(\frac{1}{\sigma^2(t)\delta_s(x)} \right) } .\]
Finally, note that $2^s\delta_s(x)\leq 4\gamma_1(A)$ by the choice of $(A_s)_{s\geq 0}$ as an almost optimal admissible sequence. 
Therefore, by distinguishing between the cases $\delta_s(x) \leq \Delta$ and $\delta_s(x)\geq \Delta$ and using that $\sigma^2(t)\geq\frac{1}{2}\Delta$, it is straightforward to verify that 
\begin{align*}
 (\ast)
 &\leq  2  \sqrt{\bar{u} \Delta}  + C_3\sigma(t)\sqrt{ \frac{\gamma_1(A)}{m} } \log^{3/2}\left(\frac{1}{\Delta} \right) \\
 &\leq C_4 \left( \sqrt{\bar{u} \Delta} + \sigma(t)\sqrt{\Delta} \right),
 \end{align*}
 where the last inequality follows from the restriction on $\Delta$  in \eqref{eq:cond.Delta}.
\end{proof}

\subsection{Controlling $\P_m(|\inr{G,x} - \inr{G,\pi_{r_u}x}|\geq \xi_u)$}
\label{sec:control.large.coord}

Now that the estimate on $|F_{m,\pi_{r_u}x}(t)-F_g(t)|$ is established, following Lemma \ref{lem:end.chain.to.everything} it is enough to show that for a well-chosen $\xi_u$, with  high probability
\[\sup_{x \in A} \P_m \left( |\inr{G,x}-\inr{G,\pi_{r_u}x}| \geq \xi_u \right) \leq C \sqrt{\bar{u} \Delta } . \]
As will become clear immediately, the natural candidate for $\xi_u$ is
\[ \xi_u= \beta \sqrt\frac{\Delta}{\bar{u} \log(1/\Delta) }, \]
where $\beta>0$ is a (small) absolute constant that is specified in what follows (in fact, $\beta=\frac{1}{2}c$ where $c$ is the constant appearing in Lemma \ref{lem:regularity} is a valid choice).
Then, because $(A_s)_{s\geq 0}$ is an almost optimal admissible sequence   and by the restriction on $\Delta$,
\begin{align*}
\|x- \pi_{r_u} x\|_{2}
\leq 2\frac{\gamma_1(A)}{2^{r_u}}
\leq \frac{c_0 \xi_u}{\log(e/\Delta)},
\end{align*}
where $c_0$ may be chosen to be sufficiently small; hence, by the gaussian tail-decay
\begin{align*}
\P \left( |\inr{G,x} - \inr{G,\pi_{r_u} x} | \geq \xi_u \right)
\leq 2\Delta
\leq 2\sqrt{ \bar{u} \Delta}.
\end{align*}
The crucial and nontrivial component is that with high probability a similar estimate holds for the empirical measure $\P_m$.

\begin{lemma}
\label{lem:end.of.chain.ratio}
	There is an absolute constant $c$ such that with  probability at least $1-2\exp(-c\Delta m)$, for every $x\in A$ and $u\in[\Delta,1-\Delta]$,
\begin{align}
\label{eq:end.ofchain}
\P_m \left( | \inr{G,x} - \inr{G, \pi_{r_u}x} | \geq \xi_u \right)
\leq 2 \sqrt{\bar{u} \Delta  }.
\end{align}
\end{lemma}

In contrast to the proof presented in Section \ref{sec:down.chain}, here there is no need to `discretize' $u$.
Clearly, \eqref{eq:end.ofchain} is equivalent to
\begin{align}
\label{eq:end.of.chaing.big.coordinates}
\left| \left\{ i  :| \inr{G_i,x- \pi_{r_u}x} | \geq \xi_u \right\} \right|
\leq  2\sqrt{\bar{u}\Delta }m
\end{align}
uniformly in $x\in A$ and $u\in[\Delta,1-\Delta]$.
Denote by $(a^\ast_i)_{i=1}^m$ the monotone non-increasing rearrangement of $(|a_i|)_{i=1}^m$ and note that for $a=(\inr{G_i,x} - \inr{G_i, \pi_{r_u}x})_{i=1}^m$,  \eqref{eq:end.of.chaing.big.coordinates} holds if $a^\ast_{ 2\sqrt{\bar{u}\Delta} m}\leq  \xi_u$; in particular, it suffices that $a^\ast_{ 2^{r_u}}\leq  \xi_u$.
Moreover, since  
\[
a_{k}^\ast \leq \left( \frac{1}{k} \sum_{i=1}^k (a_i^\ast)^2 \right)^{1/2} ,
\] 
 it is enough to show that with probability at least $1-2\exp(-c \Delta m)$, for every $u\in[\Delta,1-\Delta]$,
\begin{align}
\label{eq:end.of.chaing.big.coordinates.2.old}
\sup_{x\in A} \left(\frac{1}{2^{r_u}} \sum_{i=1}^{2^{r_u}} \left( \inr{G_i,x-\pi_{r_u} x}^\ast\right)^2 \right)^{1/2}
\leq  \xi_u.
\end{align}
By  the definitions of $\xi_u$ and $r_u$ and  the restriction on $\Delta$ in \eqref{eq:cond.Delta}, the last assertion  follows from the next lemma.

\begin{lemma}
\label{lem:end.chain.large.coord}
	There are absolute constants $c$ and  $C$ such that with probability at least $1-2\exp(-c \Delta m)$, for every $\Delta m \leq 2^r  \leq m$, 
\begin{align}
\label{eq:end.of.chaing.big.coordinates.2}
\sup_{x\in A} \left(\frac{1}{2^{r}} \sum_{i=1}^{2^{r}} \left( \inr{G_i,x-\pi_{r} x}  ^\ast\right)^2 \right)^{1/2}
\leq  C\frac{\gamma_1(A)}{2^r} \sqrt{ \log\left(\frac{1}{\Delta}\right)}.
\end{align}
\end{lemma}

The proof of Lemma \ref{lem:end.chain.large.coord} is based on a (standard) chaining argument.
For $1\leq k \leq m$, let $\mathcal{S}_k$ be   the set of $k$-sparse vectors  in the  Euclidean unit sphere,
\[\mathcal{S}_k =\{ b \in S^{m-1} : |\{i : b_i\neq 0\}| \leq  k\}.\]
In particular $(\sum_{i=1}^k (a_i^\ast)^2)^{1/2}=\max_{b\in \mathcal{S}_k}   \sum_{i=1}^m b_i a_i  $.

Moreover, we shall use the following fact, which is based on a volumetric estimate and a successive approximation argument.
\begin{lemma}
\label{lem:sphere.packing}
	There is an absolute constant $C$ such that for every $1\leq k \leq m$ there is a set $\mathcal{S}_{k}'\subset \mathcal{S}_{k}$ that satisfies   $|\mathcal{S}_{k}'|\leq \exp(C k\log(\frac{em}{k}))$ and for every $a\in \R^m$,
	\begin{align}
	\label{eq:sphere.packing} 	\left(\sum_{i=1}^{k} (a_i^\ast)^2 \right)^{1/2} \leq 2 \sup_{b\in \mathcal{S}_{k}'} \sum_{i=1}^m a_i b_i.
	\end{align}
\end{lemma}

For the sake of completeness, we include a proof of Lemma \ref{lem:sphere.packing}  in the Appendix.

\begin{proof}[Proof of Lemma \ref{lem:end.chain.large.coord}]
	For every $\Delta m\leq 2^r\leq m$, let $\mathcal{S}'_{2^r}$ be the set given in Lemma \ref{lem:sphere.packing}, and since $2^r\geq \Delta m$, we have that $|\mathcal{S}_{2^{r}}'|\leq \exp(C_1 2^{r}\log(\frac{1}{ \Delta }))$.
	Moreover, set
	\[\mathcal{E}_{r,b}=\sup_{x\in A} \frac{1}{\sqrt{2^{r}}}  \sum_{i=1}^m b_i  \inr{G_i, x- \pi_{r}x}\]
	and by \eqref{eq:sphere.packing}, 
	\begin{align}
	\label{eq:cal.E.in.proof}
	 \sup_{x\in A} \left(\frac{1}{2^{r}} \sum_{i=1}^{2^{r}} \left( \inr{G_i,x-\pi_{r} x}  ^\ast\right)^2 \right)^{1/2}
	\leq 2\sup_{b\in  	 \mathcal{S}_{2^{r}}'} \mathcal{E}_{r,b} .
	\end{align}
	We will show that for every  \emph{fixed} $\Delta m\leq 2^r\leq m$ and every \emph{fixed}  $b\in\mathcal{S}_{2^{r}}'$, with probability at least $1-2\exp(-2C_1 \cdot 2^{r} \log(\frac{1}{\Delta}))$,
	\begin{align}
	\label{eq:large.coord.proof}
	\mathcal{E}_{r,b}
\leq  C_2 \frac{\gamma_1(A)}{2^r} \sqrt{ \log\left(\frac{1}{\Delta}\right) } 
	\end{align}
	where $C_2=C_2(C_1)$ is a suitable constant.
	Once \eqref{eq:large.coord.proof} is established, it follows from the union bound---first over $b\in\mathcal{S}_{2^{r}}'$ and then over $\Delta m\leq 2^r\leq m$---that with probability at least $1-2\exp(-c_3 \Delta m)$, for every $\Delta m \leq 2^r  \leq m$, 
	\[\sup_{b\in \mathcal{S}_{2^{r}}'} \mathcal{E}_{r,b}
\leq  C_2\frac{\gamma_1(A)}{2^r} \sqrt{ \log\left(\frac{1}{\Delta}\right)}
\]
	which, by \eqref{eq:cal.E.in.proof}, completes the proof.

	\vspace{0.5em}
	To prove \eqref{eq:large.coord.proof}, let $\eta\geq1$.
	Since $\sum_{i=1}^m b_i \inr{G_i,\pi_{s+1}x-\pi_s x}$  is a gaussian random variable with mean zero and variance $\|\pi_{s+1}x-\pi_s x\|_2^2$, we have that with probability at least $1-2\exp(-\eta 2^s\log(\frac{1}{\Delta}))$,
\begin{align}
\label{eq:large.coord.lemma}
 \left| \sum_{i=1}^m b_i \inr{G_i,\pi_{s+1}x-\pi_s x} \right|
\leq \sqrt{ 2 \eta \cdot 2^{s} \log\left(\frac{1}{\Delta}\right) }  \| \pi_{s+1} x-\pi_{s} x \|_{2},
\end{align}
and clearly  $| \{ ( \pi_{s+1}x, \pi_s x ): x\in A \} |\leq \exp(2^{s+2} )$.
By the union bound over pairs $(\pi_{s+1}x,\pi_s x)$ and  $s\geq r$,  \eqref{eq:large.coord.lemma} holds uniformly  for all $x\in A$ and $s\geq r$ with probability at least
\[
1- \sum_{s\geq r}  2^{2^{s+2}} \cdot 2\exp\left( -\eta 2^s\log\left(\frac{1}{\Delta}\right)\right)
\geq 1-2\exp\left( -2C_1  \cdot 2^r \log\left( \frac{1}{\Delta } \right) \right),
\]
where the inequality is evident for a suitable choice of $\eta=\eta(C_1)$.
Denote that high probability event by  $\Omega_1(\mathbb{G})$.

Since $\gamma_1(A)\leq \gamma_1(S^{d-1})\lesssim d$, the sequence $(\|\pi_{s+1}x-\pi_sx\|_2)_{s\geq 0}$ is summable for each $x\in A$.
Thus, on the probability one  event $\Omega_2(\mathbb{G})=\{\max_{1\leq i \leq m}\|G_i\|_2<\infty\}$,
\[\mathcal{E}_{r,b}
= \sup_{x \in A} \sum_{s\geq r}  \frac{1}{\sqrt{2^{r}}} \sum_{i=1}^m b_i \inr{G_i,\pi_{s+1}x - \pi_sx} . \]
Clearly $\sqrt \frac{ 2^s }{ 2^r}  \leq \frac{ 2^s }{ 2^{r} }$ for $s\geq r$, and invoking \eqref{eq:large.coord.lemma},  it follows that on  $\Omega_1(\mathbb{G})\cap\Omega_2(\mathbb{G})$,
\begin{align*}
\mathcal{E}_{r,b}
&\leq \sup_{x\in A}   \sum_{s\geq r} \frac{2^s}{2^{r}}  \| \pi_{s+1} x-\pi_{s} x \|_2 \sqrt{  2\eta \log \left(\frac{1}{\Delta}\right)  }
\leq C_4 \frac{\gamma_1(A)}{ 2^{r}}  \sqrt{ \log \left(\frac{1}{\Delta}\right)}.
\qedhere
\end{align*}
\end{proof}

\subsection{Putting everything together: proof of Theorem \ref{thm:intro.gaussian}}
\label{sec:putting.together}

Let $\Omega(\mathbb{G})$ be the intersection of the events described in   Lemma \ref{lem:ratio.pi.r.u} and Lemma \ref{lem:end.of.chain.ratio}.
In particular, 
\[\P(\Omega(\mathbb{G}))
\geq 1-2\exp(-c \Delta m), \]
and for every realization $(G_i)_{i=1}^m\in\Omega(\mathbb{G})$ the following holds:
\begin{enumerate}[(i)]
\item 
For every $u\in U_\Delta$, $t\in\R$ and $x\in A$,
\begin{align}
\label{eq:ratio.end}
\left|F_{m, \pi_{r_u} x}(t)-F_{ g }(t)\right|
&\leq C\left( \sqrt{ \bar{u} \Delta } + \sigma(t) \sqrt{\Delta} \right).
\end{align}
\item For every $u\in[\Delta,1-\Delta]$, setting $\xi_u= \beta \sqrt\frac{\Delta}{\bar{u} \log(1/\Delta) }$ for a well-chosen absolute constant $\beta>0$, we have that
\begin{align}
\label{eq:ratio.end.dev}
\P_m\left( |\inr{G,x} - \inr{G,\pi_{r_u} x}|\geq \xi_u \right)
&\leq 2\sqrt{ \bar{u} \Delta}
\end{align}
for every $x\in A$.
\end{enumerate}
\vspace{0.5em}

All that is left to show is that for every  realization $(G_i)_{i=1}^m\in\Omega(\mathbb{G})$, $x\in A$ and $t\in\R$,  
\[ \left|F_{m,  x}(t)-F_{ g }(t)\right|
\leq C_1\left(\Delta + \sigma(t)\sqrt\Delta \right).\]
By Lemma \ref{lem:grid}, it suffices to consider $t\in  R_\Delta$.
Fix $x\in A$ and $t\in R_\Delta$, set $u=F_g(t)\in U_\Delta$ and note that $\bar{u}\sim \sigma^2(t)$.
Invoking Lemma \ref{lem:end.chain.to.everything}, 
\begin{align*}
\left|F_{m,x}(t)-F_g(t)\right| 
&\leq \sup_{t'\in[t-\xi_u, t+\xi_u]} \left|F_{ m,\pi_{r_u} x }(t')-F_{g }(t')\right|    + (F_g(t+\xi_u) -F_g(t-\xi_u) ) \\
&\qquad + \P_m\left( |\inr{G,x} - \inr{G,\pi_{r_u}x}|\geq \xi_u \right) \\
&=(1)+(2)+(3),
\end{align*}
and it follows from \eqref{eq:ratio.end.dev} that $(3)\leq 4 \sigma(t)\sqrt{\Delta}$.

Turning to $(2)$, since $\bar{u}\geq\Delta$, we have that $\xi_u\leq\beta/\sqrt{ \log(1/\bar{u})}$, and if $\beta$ is a sufficiently small absolute constant, then  by Lemma \ref{lem:regularity} 
\begin{align}
\label{eq:end.proof.something}
\sup_{t'\in[t-\xi_u, t+\xi_u]} |F_g(t')-F_g(t)|
\leq C_2\sigma(t)\sqrt{\Delta};
\end{align}
in particular, $(2)\leq 2C_2\sigma(t)\sqrt{\Delta}$.

Finally, to estimate $(1)$, note that $\sigma^2(t) \geq \frac{1}{2}\Delta$,  and \eqref{eq:end.proof.something} implies that
\[\sup_{t'\in[t-\xi_u,t+\xi_u]} \sigma^2(t') 
\leq \sigma^2(t) +  C_3\sigma(t)\sqrt{\Delta}
\leq (1+2C_3)\sigma^2(t). \]
Thus, by  \eqref{eq:ratio.end},  $(1) \leq C_4\sigma(t) \sqrt{\Delta}$.
That completes the proof of Theorem \ref{thm:intro.gaussian}.
\qed

\section{Proof of Proposition \ref{prop:intro.sudakov}}

In what follows, $(\varepsilon_i)_{i=1}^m$  are independent, symmetric Bernoulli  random variables (that is, they take the values $\pm 1$ with probability $\frac{1}{2}$) that are independent of $(G_i)_{i=1}^m$.
The expectation  with respect to $(\varepsilon_i)_{i=1}^m$ is denoted by $\E_\varepsilon$, and the expectation with respect to  $(G_i)_{i=1}^m$ is denoted by  $\E_\mathbb{G}$.
Similar conventions apply to probabilities. 

\begin{lemma}
\label{lem:sudakov.expectation}
	There are absolute constants $C$ and $c$ such that the following holds.
	For every $\delta>0$, if  $m\geq C\max\{\frac{1}{\delta^2}, \frac{\log \mathcal{N}(A,\delta B_2^d)}{\delta} \}$ then
	\begin{align}
	\label{eq:sudakov.expectation}
	\E \sup_{x\in A}  \left| F_{m,x}(0) - F_{g}(0) \right|
	\geq c \sqrt{\frac{\delta \log \mathcal{N}(A,\delta B_2^d)}{m}}  - \frac{1}{\sqrt m}.
	\end{align}
\end{lemma}

For the rest of this section, fix $\delta$ and $m$ as in Lemma \ref{lem:sudakov.expectation}.
We may assume without loss of generality that $\delta \leq 2$ because $\log\mathcal{N}(A,\delta B_2^d)=0$ for $\delta>2$ and  \eqref{eq:sudakov.expectation} is trivially satisfied.
By the obvious relation between packing and covering numbers, there is a set $A'\subset A$ with cardinality $\mathcal{N}(A,\delta B_2^d)$ that is $\delta/2$-separated with respect to  the Euclidean distance; thus, $\|x-y\|_2\geq \delta/2$ for every distinct $x,y\in A'$.

A standard (de-)symmetrization argument (see \cite{gine1984some} or, e.g., \cite[Lemma 11.4]{boucheron2013concentration}) implies that
\begin{align*}
	\E\sup_{x\in A}  \left| F_{m,x}(0) - F_{g}(0) \right|
	&\geq  \E \sup_{x\in A} \frac{1}{m}\sum_{i=1}^m \left( 1_{(-\infty,0]}(\inr{G_i,x})  - \frac{1}{2}  \right) \\
	&\geq \frac{1}{2}\E_{\mathbb{G}}\E_{\varepsilon} \sup_{x\in A} \frac{1}{m}\sum_{i=1}^m \varepsilon_i 1_{(-\infty,0]}(\inr{G_i,x}) 
	- \frac{1}{\sqrt m}.
\end{align*}
Also, observe that  for every $x,y\in S^{d-1}$, 	
\begin{align}
\label{eq:separated.probability}
 \P\left( \inr{ G,x}\leq 0 \, , \, \inr{G,y}> 0 \right) \geq c_0 \|x-y\|_{2}.
\end{align}
Indeed, just as in the proof of Lemma \ref{lem:sym.diff}, the probability in \eqref{eq:separated.probability} is  the gaussian measure of a cone with angle  that is proportional to $\|x-y\|_2$.

Therefore, a standard binomial estimate combined with the union bound leads to the following observation: 
With  $\P_{\mathbb{G}}$-probability at least $1-\exp(-c_1\delta m)$, for every distinct $x,y\in A'$,
\begin{align}  
\label{eq:123}
 \left| \left\{ i : \inr{G_i,x} \leq 0  \, , \, \inr{G_i,y}> 0 \right\} \right|
\geq \frac{c_0}{4} \delta m. 
\end{align} 

The final ingredient needed for the proof of Lemma \ref{lem:sudakov.expectation} is a Sudakov-type bound for Bernoulli processes established by Talagrand (see, e.g., \cite{ledoux1991probability}, and \cite{ latala2014sudakov,mendelson2019generalized} for the formulation used here).

\begin{theorem}
\label{thm:bernoulli.sudakov}
	There are absolute constants $C$ and $c$ such that the following holds.
	Let $\mathcal{V}\subset \R^m$, set $2\leq p\leq m$ and put $\eta\leq \sup_{v\in\mathcal{V}} \|v\|_2$.
	If $|\mathcal{V}|\geq \exp(p)$ and the set $\{\sum_{i=1}^m \varepsilon_i v_i : v\in \mathcal{V}\}$ is $C \sqrt p \cdot \eta$-separated in $L_p(\varepsilon)$, then
	\[ \E \sup_{v\in\mathcal{V}} \sum_{i=1}^m \varepsilon_i v_i
	\geq c \sqrt{p} \cdot \eta . \]
\end{theorem}

\begin{proof}[Proof of Lemma \ref{lem:sudakov.expectation}]
	Denote by $\Omega(\mathbb{G})$ the event in which \eqref{eq:123} holds for every distinct $x,y\in A'$ and fix $(G_i)_{i=1}^m\in\Omega(\mathbb{G})$.
	Let $C$ be the constant appearing in Theorem \ref{thm:bernoulli.sudakov} and observe that  for every  $2\leq p \leq c_1\delta m $ the set 
	\[ \mathcal{V}= \left\{ \left( 1_{(-\infty,0]}(\inr{G_i,x})  \right)_{i=1}^m : x\in A' \right\} \]
	is $C \sqrt{p}\cdot \sqrt{\delta m}$-separated in $L_p(\varepsilon)$.
	Indeed, by a result due to Hitczenko \cite{hitczenko1993domination}, for every $p\geq 2$ and $v\in \R^m$,
	\begin{align}
	\label{eq:bernoulli.vector.lp.norm}
	\left\| \sum_{i=1}^m \varepsilon_i v_i \right\|_{L_p}
	\geq c_2 \left( \sum_{i=1}^p v^\ast_i + \sqrt{p}\left( \sum_{i=p+1}^m (v^\ast_i)^2 \right)^{1/2} \right).
	\end{align}
	Moreover, by the definition of $\Omega(\mathbb{G})$, for every  distinct $v,w\in\mathcal{V}$ we have that
	\begin{align}
	\label{eq:bernoulli.vector.lp.norm.no2}
	\left|\left\{ i  : v_i\neq w_i \right\}\right| \geq c_3\delta m;
	\end{align}
	in particular $(v-w)^\ast_i =1$ for $1\leq i \leq c_3\delta m$ and $\sum_{i\geq \frac{1}{2}c_3\delta m}^m ((v_i-w_i)^\ast)^2 \geq \frac{1}{2}c_3\delta m$.

	Next, since  $m\geq C_4$ and $\delta\geq 1/\sqrt{m}$, it is evident that  $\P(\Omega(\mathbb{G}))\geq \frac{1}{2}$.
	If 
	\[ p=\log \mathcal{N}(A,\delta B_2^d)
	\text{ and } \eta=c_5\sqrt{\delta m},\] then the set $\mathcal{V}$ satisfies the assumption from Theorem \ref{thm:bernoulli.sudakov}.
	Hence,
	\begin{align}
	\label{eq:lower.bound.expectation}
	 \E_\varepsilon \sup_{x\in A'} \sum_{i=1}^m \varepsilon_i  1_{(-\infty,0]}(\inr{G_i,x}) 
	\geq c_6\sqrt{  \log \mathcal{N}(A,\delta B_2^d)} \cdot \sqrt{\delta m   }
	\end{align}
	and as $\P(\Omega(\mathbb{G}))\geq \frac{1}{2}$, by Fubini's theorem
	\[ \E_{\mathbb{G}}\E_\varepsilon \sup_{x\in A'} \frac{1}{m}\sum_{i=1}^m \varepsilon_i  1_{(-\infty,0]}(\inr{G_i,x}) 
	\geq \frac{c_6}{2} \sqrt \frac{ \delta \log \mathcal{N}(A,\delta B_2^d) }{m} .
	\qedhere
\]
\end{proof}

\begin{proof}[Proof of Proposition \ref{prop:intro.sudakov}]
	We assume without loss of generality that $\delta\leq 2$, otherwise the statement is trivially satisfied.
	
	First note that we may assume that
	\[(\ast)=\sqrt{\frac{\delta \log \mathcal{N}(A, \delta B_2^d) }{ m} }\geq  \frac{C_0}{ \sqrt m}\]
	 for a suitable constant $C_0$ that is  specified in what follows.
	Indeed, since  $\P(\inr{G,x}\leq 0)=\frac{1}{2}$ for every $x\in A$, it follows from the optimality of the estimate in  Theorem \ref{thm:intro.ratio.single} (see Proposition \ref{prop:ratio.probability}) that if $m$ is sufficiently large, then 
	\[ \P\left( |F_{m,x}(0)-F_g(0)| \geq \frac{c_1}{\sqrt m} \right)\geq 0.9.\]
	
	Moreover, by Lemma \ref{lem:sudakov.expectation}, 
	\begin{align*}
	\E \sup_{x\in A} |F_{m,x}(0)-F_g(0)| 
	&\geq c_2(\ast) - \frac{1}{\sqrt m} 
	\geq \frac{c_2}{2}(\ast) + \frac{2}{\sqrt m},
	\end{align*}
	where the last inequality holds if $(\ast)\geq \frac{6}{c_2\sqrt m}$.
	Finally, by the bounded difference inequality (see, e.g., \cite[Theorem 6.2]{boucheron2013concentration}), for every $\lambda>0$, with probability at least $1-\exp(-2 \lambda^2m)$, 
	\begin{align*}
	\sup_{x\in A} |F_{m,x}(0)-F_g(0)| 
	&\geq \E \sup_{x\in A} |F_{m,x}(0)-F_g(0)| - \lambda\\
	&\geq \frac{ c_2}{2} (\ast) + \frac{2}{\sqrt m} - \lambda.
	\end{align*}
	Thus, the proof is completed by setting  $\lambda=\frac{2}{\sqrt m}$.
\end{proof}

\section{The Wasserstein distance}
\label{sec:Wasserstein}

The estimate in Theorem \ref{thm:intro.gaussian} implies that the empirical distribution function $F_{m,x}$ and the true distribution function $F_g$ are close in a rather strong sense.
In particular,  as we will demonstrate, they are close in the  \emph{Wasserstein distance} $\mathcal{W}_2$:

\begin{definition}
For two distribution functions $F$ and $H$ on the real line with finite second moments, the $\mathcal{W}_2$ Wasserstein distance between $F$ and $H$ is given by
\[ \mathcal{W}_2(F,H)
=\inf_\pi \left( \int_{\R\times\R} (x-y)^2 \,\pi(dx,dy) \right)^{1/2} ,\]
where the infimum is taken over all probability measures $\pi$ with first marginal $F$ and second marginal $H$.
\end{definition}

We refer to \cite{bartl2022structure,bobkov2019one,figalli2021invitation,villani2021topics} and the references therein for more information on \emph{optimal transport} and (statistical aspects of) Wasserstein distances.

Thanks to the estimate in Theorem \ref{thm:intro.gaussian}, one can show the following:

\begin{tcolorbox}
\begin{theorem}
\label{thm:Wasserstein}
	There are absolute constants $C_1,c_2,C_3$ such that the following holds.
	Let $A \subset S^{d-1}$ be symmetric. 
	Set $m\geq \gamma_1(A)$ and put
	\[
	\Delta \geq C_1 \frac{\gamma_1(A)}{m} \log^3\left(\frac{em}{\gamma_1(A)}\right) .
	\]
	Then with probability at least $1-2\exp(-c_2\Delta m)$,	 
	\[ \sup_{x\in A } \mathcal{W}_2\left( F_{m,x} , F_g \right) 
	\leq C_3 \sqrt{\Delta \log\left(\frac{1}{\Delta}\right) }. \]
\end{theorem}
\end{tcolorbox}

\begin{remark}
Theorem \ref{thm:Wasserstein} generalizes the main result in \cite{bartl2022structure} in the gaussian case---from $A=S^{d-1}$ to an arbitrary subset $A\subset S^{d-1}$.
\end{remark}

The proof of Theorem \ref{thm:Wasserstein} relies on the fact that if the empirical and  true distribution functions are `close' (as in Theorem \ref{thm:intro.gaussian}), then their inverse functions are close in a suitable sense as well.
To that end, let  $\Omega(\mathbb{G})$ be the event in which, for every $t\in \R$,
\begin{align}
\label{eq:wasserstein.ratio}
\sup_{x\in A}|F_{m,x}(t) - F_g(t) |
\leq \Delta +  \sigma(t)\sqrt{\Delta}.
\end{align}
Thus, if $\Delta$ and $m$ satisfy the conditions in Theorem \ref{thm:Wasserstein}, then by Theorem \ref{thm:intro.gaussian}, $\P(\Omega(\mathbb{G}))\geq 1-2\exp(-c\Delta m)$.

For a distribution function $H$, denote by $H^{-1}(u)=\inf\{ t\in\R : H(t)\geq u\}$ its right-inverse function.

\begin{lemma}
\label{lem:inverse}
	There is an absolute constant $C$ such that the following holds.
	For every  realization of $(G_i)_{i=1}^m\in\Omega(\mathbb{G})$ and  $u\in[C\Delta,1-C\Delta]$,
	\[ F_{m,x}^{-1}(u) \in \left[ F_g^{-1}(u-4\sqrt{\bar{u}\Delta}) , F_g^{-1}(u+4\sqrt{\bar{u}\Delta})  \right].\]
\end{lemma}
\begin{proof}
	We shall only verify the upper bound, as the lower one follows from the same argument.
	Setting $t=F_g^{-1}(u+4\sqrt{\bar{u}\Delta})$, it  suffices to show that $u< F_{m,x}(t)$.
	Using \eqref{eq:wasserstein.ratio},
	\begin{align*}
\nonumber
	F_{m,x}(t) 
	&\geq F_g(t) - \Delta -\sigma(t)\sqrt\Delta \\
	&= u + 4\sqrt{\bar{u}\Delta} - \Delta - \sigma(t)\sqrt\Delta
	=(\ast).
	\end{align*}
	Moreover, if  $\bar{u}\geq C \Delta$ and $C$ is sufficiently large, we have that $4\sqrt{\bar{u}\Delta}\leq \frac{1}{2}\bar{u}$; thus  $\Delta \leq \sigma^2(t)\leq 2 \bar{u}$ and $(\ast)>u$.
\end{proof}

\begin{proof}[Proof of Theorem \ref{thm:Wasserstein}]
Let $\kappa$ be a suitable absolute constant that is specified in what follows.
Set $U=[0,1]\setminus[\kappa\Delta,1-\kappa\Delta]$ and fix a realization $(G_i)_{i=1}^m\in\Omega(\mathbb{G})$.
By a standard representation of the Wasserstein distance in $\R$ (see, e.g., \cite[Theorem 2]{ruschendorf1985wasserstein}),
\begin{align*}
&\mathcal{W}_2^2(F_{m,x},F_g)
=\int_0^1 \left( F_{m,x}^{-1}(u) - F_g^{-1}(u) \right)^2 \,du \\
&\leq \int_{\kappa\Delta}^{1-\kappa\Delta} \left( F_{m,x}^{-1}(u) - F_g^{-1}(u) \right)^2 \,du + 2\int_U \left( F_{m,x}^{-1}(u) \right)^2\,du + 2\int_U \left( F_{g}^{-1}(u) \right)^2\,du \\
&=(1) + (2) + (3). 
\end{align*}
Using the gaussian tail-decay, $(3) \leq C_1(\kappa)\Delta\log(\frac{1}{\Delta})$.
It remains to estimate $(1)$ and $(2)$ uniformly in $x$.

If $\kappa$ is sufficiently large, an application of Lemma \ref{lem:inverse}  shows that
\[ (1) \leq \int_{\kappa\Delta}^{1-\kappa\Delta} \left( F_g^{-1}(u+ 4\sqrt{\bar{u}\Delta})  - F_g^{-1}(u- 4\sqrt{\bar{u}\Delta}) \right)^2 \,du.\]
Recall that $f_g(F_g^{-1}(v))\sim \bar{v}\sqrt{2 \log(1/\bar{v})}\geq \bar{v}$ for every $v\in(0,1)$.
Moreover,  if $\kappa$ is sufficiently large, for $u\in[\kappa\Delta,1-\kappa\Delta]$ we have that $4\sqrt{\bar{u}\Delta} \leq \frac{1}{2}\bar{u}$, and in particular $\frac{1}{2}\bar{u}\leq \bar{ v} \leq 2 \bar{u}$ for every $v\in[u-4\sqrt{\bar{u}\Delta} , u+4\sqrt{\bar{u}\Delta}]$.
By the fundamental theorem of calculus,
\[
 F_g^{-1}(u+ 4\sqrt{\bar{u}\Delta})  - F_g^{-1}(u- 4\sqrt{\bar{u}\Delta})
 =\int_{u-4\sqrt{\bar{u}\Delta}}^{u+4\sqrt{\bar{u}\Delta}}\frac{dv}{f_g(F_g^{-1}(v))}
\leq C_2 \sqrt\frac{\Delta }{\bar{u}}  \]
and  $(1)\leq C_3(\kappa) \Delta \log(\frac{1}{\Delta})$.

Finally,  observe that for every $1\leq i\leq m$,
\begin{align}
\label{eq:F.inverse}
F_{m,x}^{-1}(u)=\inr{G_i,x}^\sharp \quad\text{for }u\in\left(\frac{i-1}{m},\frac{i}{m}\right];
\end{align} hence
\[ (2)  
\leq \sup_{x\in A}\frac{1}{m} \sum_{i=1}^{2\kappa\Delta m} (\inr{G_i,x}^\ast)^2 
=(4). \]
A chaining argument (e.g.\ following the  same path as presented in Section \ref{sec:control.large.coord}) shows that with probability at least $1-2\exp(-c_4\Delta m)$, $(4)\leq C_5(\kappa)\Delta\log(\frac{1}{\Delta})$.
\end{proof}

\section{Other proofs}

\subsection{Proof of Theoren \ref{thm:wasserstein.matrix}}

	For $1\leq i\leq m$ set $\eta_i=m \int_{(i-1)/m}^{i/m} F_{g}^{-1}(u)\,du$ and write
	\begin{align*}
	\left( \frac{1}{m}\sum_{i=1}^m \left( (\Gamma x^\sharp)_i -\lambda_i \right)^2 \right)^{1/2}
	&\leq \left( \frac{1}{m}\sum_{i=1}^m \left( (\Gamma x^\sharp)_i -\eta_i \right)^2 \right)^{1/2} + \left( \frac{1}{m}\sum_{i=1}^m \left(\eta_i-\lambda_i \right)^2 \right)^{1/2} \\
	&=(1)+(2).
	\end{align*}
	
	To estimate $(1)$,  first note that by \eqref{eq:F.inverse},
	 $(\Gamma x)^\sharp_i= m \int_{(i-1)/m}^{i/m} F_{m,x}^{-1}(u)\,du $.
	 An application of Jensen's inequality shows that 
	\begin{align*}
	(1)^2
	&\leq \sum_{i=1}^m \int_{\frac{i-1}{m}}^{\frac{i}{m}} \left( F_{m,x}^{-1}(u) - F_g^{-1}(u)  \right)^2 \,du 
	=\mathcal{W}_2^2(F_{m,x},F_g)
	\end{align*}
	and in the high probability event in which the assertion of Theorem \ref{thm:Wasserstein} holds,
	\[\sup_{x\in A}\mathcal{W}_2(F_{m,x},F_g)\leq C_1\sqrt{\Delta\log\left(\frac{1}{\Delta} \right)}.\]
	
	Next, let us show that $(2)\leq C \sqrt{\frac{\log(m)}{m}}$; that will conclude the proof because $\Delta \geq\frac{1}{m}$.
	To show that $|\eta_i-\lambda_i|$ is small, we focus on $1\leq i \leq \frac{m}{2}$; the estimate when $\frac{m}{2}<i\leq m$ follows from the same arguments and is omitted.
	
	By the gaussian tail-decay, $|\lambda_1|, |\eta_1|\leq C_2\sqrt{\log(m)}$.
	Moreover, since $f_g(F_g^{-1}(u))\sim u\sqrt{\log(1/u)}$ for $u\leq\frac{1}{2}$, it is evident from a Taylor expansion that for $2\leq i \leq\frac{m}{2}$,
	\begin{align*}
	\lambda_i - \eta_i
	&=m \int_{\frac{i-1}{m}}^{\frac{i}{m}} \left( F_g^{-1}\left(\tfrac{i}{m} \right)-F_{g}^{-1}(u)  \right) \,du \\
	&\leq C_3 m \int_{\frac{i-1}{m}}^{\frac{i}{m}} \frac{\frac{i}{m} - u}{ \frac{i-1}{m}\sqrt{ \log (\frac{m}{i-1}) } }\, du
	\leq  \frac{C_4}{i \sqrt{ \log(\frac{m}{i}) }}. 
	\end{align*}	
	This clearly implies the wanted estimate on $(2)$.
	\qed


\subsection{Proof of Example \ref{ex:intro.density}}

	Let $d\geq 2$ to be specified in what follows, set $0<\delta\leq  \frac{1}{2\log (d)}$, and put
	\[ A = \left\{ \sqrt{1-\delta^2} \cdot e_1 + \delta \cdot e_k : k=2,\dots, d\right\}. \]
	Clearly $A\subset S^{d-1}$, its Euclidean diameter  is bounded by $2\delta$, and it satisfies that $\gamma_1(A)\leq 2 \log(d) \delta\leq 1$.
	Let us show that with probability at least $0.9$, 
	\[\sup_{x\in A} |F_{m,x}(-1)-F_x(-1)|\geq \frac{1}{10}.\]
	
	To that end, denote by $I\subset\{1,\dots,m\}$ the random  set consisting of indices corresponding to the   $0.4  m$  smallest coordinates of $(\inr{X_i,e_1})_{i=1}^m$.
	As $\P(\inr{X,e_1}\leq -1)=\frac{1}{2}$, it follows from Markov's inequality that if $m\geq C_1$,  the event
	\begin{align*}
\Omega_1(\mathbb{X})
&=\{\inr{X_i,e_1}=-1 \text{ for every } i\in I\} 
\end{align*}
	has probability at least  $0.99$.	
	
	Since $I$ depends only on $(\inr{X_i,e_1})_{i=1}^m$ and $X$ has independent coordinates,  it is enough to show that conditionally on $\Omega_1(\mathbb{X})$, the event
\[\Omega_2(\mathbb{X})=\left\{\text{there is } 2\leq k\leq d \text{ such that }\inr{X_i,e_k}=-1 \text{ for every } i\in I \right\}\]
has probability at least $0.99$.
Indeed, for every $2\leq k\leq d$, $x_k=\sqrt{1-\delta^2}e_1 +  \delta e_k \in A$, and on the intersection of the two events,
\[\inr{X_i,x}= - \sqrt{1-\delta^2} - \delta<-1 
\quad\text{for every } i\in I;\]
thus,   $F_{m,x}(-1) \geq \frac{|I|}{m}=0.4$ while  $F_x(-1)=\frac{1}{4}$, as required.

	Conditionally on $\Omega_1(\mathbb{X})$, for every $k=2,\dots, d$, with probability $2^{-|I|}$, $\inr{X_i,e_k}=-1 $  for every $i\in I$.
	Hence the conditional probability of $\Omega_2(\mathbb{X})$ is 
\[ 
1 - \left( 1- 2^{-|I|} \right)^{d-1}
\geq  1- \exp\left(-\tfrac{1}{2} (d-1) 2^{-|I|}\right)
\geq 0.99, \]
	where the last inequality holds if $d\geq C_2\exp(m)$.
	\qed
	
\appendix
\section{On the scale dependent DKW inequality}
\label{app:ratio}

Let us outline  the claims related to the optimality of Theorem \ref{thm:intro.ratio.single} mentioned in the introduction.
Their proofs  can be found in  \cite{bartl2023variance}.

We begin with the fact that the restriction $\Delta\gtrsim\frac{\log\log m}{m}$ in Theorem \ref{thm:intro.ratio.single} is optimal.

\begin{theorem}
\label{thm:ratio.LIL}
	Assume that there is a positive sequence $(\Delta_m)_{m=1}^\infty$ and numbers $\alpha,\beta>0$ for which the following holds:
	For every $\Delta\geq \Delta_m$, with probability at least $1-2\exp(-\alpha \Delta m)$, for every $t\in\R$, 
	\[ |\P_m(X\leq t) - \P(X\leq t)|
	\leq \beta \left( \sigma(t)\sqrt\Delta + \Delta \right).\]
	Then there are constants $C$ and $c$ that depend only on $\alpha$ and $\beta$ that satisfy that for every $m\geq C$, 
	\[\Delta_m \geq c\frac{\log\log m}{m}.\]
\end{theorem}

In addition to the optimality of the restriction on $\Delta$,  the probability estimate  in Theorem \ref{thm:intro.ratio.single} is optimal as well.
In fact, it cannot be improved even for a fixed  $t\in\R$.

\begin{proposition} 
\label{prop:ratio.probability}
	There are absolute constants $C_1,c_2,c_3$  such that the following holds.
	For every $\Delta\geq \frac{C_1}{m}$ and  $t\in\R$ that satisfies $\sigma^2(t)\geq\Delta$, with probability at least $2\exp(-c_2\Delta m)$,
	\[ \left|  \P_m(X\leq t) - \P(X\leq t) \right|
	\geq c_3 \sigma(t) \sqrt{\Delta} .\]
\end{proposition}

\begin{remark}
The focus in Proposition \ref{prop:ratio.probability} is on the range $\sigma^2(t)\geq\Delta$ because in that range the behaviour of $|\P_m(X\leq t) - \P(X\leq t)|$ is governed by cancellations. 
For smaller values of $\sigma^2(t)$ (relative to $\Delta$) no cancellations are to be expected.
If $t$ satisfies $\P(X\leq t)\ll \Delta$, then
\[ | \P_m(X\leq t)-\P(X\leq t)| 
\sim \max \left\{ \P_m(X\leq t), \P(X\leq t) \right\} .\]
We refer to \cite{bartl2023variance} for versions of Theorem \ref{thm:ratio.LIL} and Proposition \ref{prop:ratio.probability} in the `no cancellations' range.
\end{remark}

\section{Proof of Lemma \ref{lem:sphere.packing}}

For every non-empty set $E\subset\{1,\dots,m\}$ put
\[S_E =\{ b\in S^{m-1} : b_i=0 \text{ for } i\notin E\},\]
and observe that $\mathcal{S}_k = \bigcup_{E : |E|\leq k} S_E$ for every $1\leq k \leq m$.
Each $S_{E}$ is a sphere of dimension $|E|$, and by a standard volumetric estimate (see, e.g., \cite[Corollary 4.1.15]{artstein2015asymptotic}), there is a subset $S'_E \subset S_E$ that is a $\frac{1}{2}$-cover of $S_E$ with cardinality at most $\exp(C_1 |E|)$.
Note that $|\{ E : |E|= \ell \}|=\binom{n}{\ell} \leq \exp(\ell \log(\frac{em}{\ell}))$ for all $1\leq \ell\leq m$, thus  the set $\mathcal{S}_k'=\bigcup_{E : |E|\leq k} S_E'$ has cardinality at most $\exp(C_2 k \log(\frac{em}{k}))$.

Fix $a\in\R^m$ and $1\leq k\leq m$ and recall that $(\sum_{i=1}^{k} (a_i^\ast)^2 )^{1/2}  = \sup_{b\in \mathcal{S}_k} \inr{a,b}$.
Let $E\subset\{1,\dots,m\}$ and  $b\in S_E$, let $\pi b\in S_E'$ satisfy that $\|b-\pi b\|_2\leq \frac{1}{2}$ and clearly, 
\[ \inr{ a,b }
=  \inr{ a, \pi b  }  +  \inr{ a,b-\pi b } 
\leq \sup_{b'\in S_E'} \inr{a, b' } + \frac{1}{2} \sup_{b'' \in S_E } \inr{a,b''} . \]
Therefore,
\[\sup_{b\in S_E} \inr{a,b}\leq 2 \sup_{b'\in S_E'} \inr{a, b' } \]
 and the lemma follows by taking the maximum over $E$ as well.
\qed

\vspace{1em}
\noindent
\textsc{Acknowledgements:}
Daniel Bartl is grateful for financial support through the Austrian Science Fund (grant doi: 10.55776/ESP31 and 10.55776/P34743).
The authors would like to thank the referee for a careful reading and helpful suggestions.

\bibliographystyle{plain}

\end{document}